\documentclass[10pt]{article}
\usepackage{amsmath, amsfonts, amsthm, amssymb, amscd, a4wide}
\usepackage[mathcal]{euscript} 
\usepackage[applemac]{inputenc} 
\usepackage{hyperref}

\usepackage{dsfont, verbatim}
\usepackage{mathrsfs}  
\newtheorem{theorem}{Theorem}[section]
\newtheorem{proposition}[theorem]{Proposition} 
\newtheorem{lemma}[theorem]{Lemma}
 
\newtheorem{definition}[theorem]{Definition} 
  
\theoremstyle{definition}
  
\newtheorem{remark}[theorem]{Remark}  
\numberwithin{equation}{section} 
\newcommand \be           {\begin{equation}}
\newcommand \ee            {\end{equation}}

\newcommand \Ecal           {\mathcal{E}}

\newcommand \RR           {\mathbb{R}}
\newcommand \NN           {\mathbb{N}}

\newcommand \del           \partial
\newcommand \eps            \epsilon

\newcommand{\ve}{\varepsilon}
\newcommand{\vp}{\varphi}

\newcommand{\ess}{\,{\rm ess}}
\newcommand \WW     {\mathcal{W}}
\DeclareMathOperator    \sgn {sgn}
\DeclareMathOperator    \TV  {TV} 

\DeclareMathOperator\dive {div}
\newcommand{\benum}{\begin{enumerate}}
\newcommand{\eenum}{\end{enumerate}}
\newcommand \ws     {\mathrel{\mathop{\rightharpoonup}\limits^{*}}}

\def\XXint#1#2#3{{\setbox0=\hbox{$#1{#2#3}{\int}$}
\vcenter{\hbox{$#2#3$}}\kern-.5\wd0}}

\newcommand{\loc}{\mathrm{loc}}

\begin{document}

\title{The obstacle\,-\,mass constraint problem for \\
hyperbolic conservation laws. Solvability}

\author{Paulo Amorim$^{1,3}$, Wladimir Neves$^1$, and Jos\'e Francisco Rodrigues$^2$}

\date{}
\maketitle

\footnotetext[1]{Instituto de Matem\'atica, Universidade Federal do Rio de Janeiro, Av. Athos da Silveira Ramos 149,
Centro de Tecnologia - Bloco C,
Cidade Universit\'aria - Ilha do Fund\~ao,
Caixa Postal 68530, 21941-909 Rio de Janeiro,
RJ - Brasil.}

\footnotetext[2]{
CMAF+IO, F\_Ciências,
Universidade de Lisboa,
P-1749-016 Lisboa, Portugal.
}

\footnotetext[3]{Corresponding author. E-mail: {paulo@im.ufrj.br}}

\begin{abstract} 
In this work we introduce the obstacle-mass constraint problem for a multidimensional scalar hyperbolic conservation law.
We prove existence of an entropy solution to this problem by a penalization/viscosity method. The mass constraint introduces a nonlocal Lagrange multiplier in the penalized equation, giving rise to a nonlocal parabolic problem. We introduce a compatibility condition relating the initial datum and the obstacle function which ensures global in time existence of solution. This is not a smoothness condition, but relates to the propagation of the support of the initial datum. 


\end{abstract}



\section{Introduction}

We consider the Cauchy problem for a hyperbolic conservation law on $(t,x) \in(0,T)\times\RR^d$,
\be
\label{010}
\aligned
&H(u) \equiv \del_t u + \dive f(u) = 0,
\\
& u(0,x)= u_0(x), \qquad x \in \RR^d,
\endaligned
\ee
under both restrictions 
\be
\label{011}
\aligned
0\le u(t,x) \le \theta(t,x)
\endaligned
\ee
and
\be
\label{012}
\aligned
\int_{\RR^d} u(t,x) \,dx = 1, \quad t\ge 0.
\endaligned
\ee
Here, $\theta(t,x)$ is a given obstacle function, $f$ is the flux function which is supposed smooth, and the Cauchy datum $u_0$ is such that $0 \le u_0(x) \le \theta(0,x),$ with
$\int_{\RR^d} u_0(x) \,dx = 1$. In all that follows, every solution $u$ of the various problems we will consider will be nonnegative, this being a consequence of 
the non negativeness of the initial datum and
the properties of the 
operator $H$.

Even without the mass constraint \eqref{012}, some sense must be given to the hyperbolic problem \eqref{010} under the obstacle constraint $u \le \theta$. This was done mainly by L\'evi in a series of works \cite{Levi3,Levi1,Levi2}, in the case of a Dirichlet problem, in which a viscous approximation was introduced with a penalization term enforcing a constraint of type $u\le \theta$.

One way to understand L\'evi's approach is to observe that (formally at least) a solution $u$ to the obstacle problem $H(u)=0$, $u\le \theta$ actually corresponds to a couple $(u,\mu)$ verifying
\be
\label{015}
\aligned
\del_t u + \dive f(u) = \mu,
\endaligned
\ee
with
\[
\aligned
\mu = - H(\theta)^- \chi_{\{u=\theta\}},
\endaligned
\]
where we define the positive and negative parts as $v^+:= \ess \sup \{v, 0 \}$, $v^-= (-v)^+$,
and $H$ is the operator defined in~\eqref{010}. The motivation for the above 
equation can be found in Remark 4.2 in \cite{JFR2} for the linear case.
In fact, equation \eqref{015} means that $u$ must solve the equation $H(u) =0$ wherever $u$ does not coincide with $\theta$. On the other hand, on the coincidence set $\{ u= \theta\}$ where the Lagrange multiplier $\mu$ is active, formally one has $H(\theta) = - H^-(\theta)$, which is to say, $H(\theta)\le 0$. 

However, such a solution, while verifying $u\le\theta$, does not conserve mass. This reduces the applicability of that approach to problems where mass conservation is important, such as in porous media models with saturation arising in petroleum engineering and crowd or traffic dynamics (see, however, \cite{BB} for an application of an obstacle problem enforcing mass loss). More examples of domains where hyperbolic obstacle problems may be applicable can be found in \cite{Oil} and the references in \cite{Levi1}. Other references on hyperbolic obstacle problems include \cite{Berthelin,Puel,JFR1,JFR2}, although we could cite many others. 
For an introduction to classical obstacle problems, we address the reader to the book of Kinderlehrer and Stampacchia \cite{DKGS}, and also Rodrigues 
 \cite{JFRBook}.


It is clear that a solution to \eqref{010}--\eqref{012}, taken in a na\" ive sense, may not exist if the obstacle $\theta$ is reached. Indeed, in that case, there are two mutually exclusive effects taking place: on the one hand, the evolution equation $H(u) =0$ naturally conserves the total mass; on the other hand, the presence of the obstacle leads to mass loss. In this work we propose a mechanism designed to reconcile these two contradictory aspects.
One classical way in which an integral constraint like the unit integral condition in \eqref{012} may be enforced, is to introduce a Lagrange multiplier into the equation \eqref{010}, see for instance Caffarelli and Lin \cite{CaffLin} for a related problem. Taking this approach here, our problem without obstacle may be posed as follows: we look for a pair $(u,\lambda)$, with $\lambda(t)$ a function of $t$ alone, such that $u$ and $\lambda$ satisfy 
\be
\label{10}
\aligned
&\del_t u + \dive f(u) = \lambda(t) u,
\\
& u(0,x)= u_0(x),
\endaligned
\ee
where the Lagrange multiplier $\lambda(t)$ ensures that
\[
\aligned
\int_{\RR^d} u(t,x) \,dx = 1, \quad t\ge 0.
\endaligned
\]
To our knowledge, this procedure is completely new for scalar conservation laws. Moreover, we require that $u$ satisfy both restrictions \eqref{011} and \eqref{012}.
Thus, even while respecting the obstacle condition, the solution 
$u$ conserves the total mass, which is physically relevant for real applications.  

Formally, and in agreement with \eqref{015} and \eqref{10}, the solution $u$ of \eqref{010}--\eqref{012} should verify 
\be
\label{15}
\aligned
\del_t u + \dive f(u) = -H(\theta)^- \chi_{\{u=\theta\}} + u \int_{\RR^d} H(\theta)^- \chi_{\{u=\theta\}} \,dx.
\endaligned
\ee
In fact, setting $\lambda(t) = \int_{\RR^d} H(\theta)^- \chi_{\{u=\theta\}} \,dx$ and integrating \eqref{15} on $\RR^d$ one finds formally
\[
\aligned
\frac d{dt}\int_{\RR^d}  u \,dx =  \lambda(t) \Big(\int_{\RR^d} u \,dx -1 \Big),
\endaligned
\]
which yields the conservation of mass $\int_{\RR^d} u \,dx =1$ for $t>0$, as long as $\int_{\RR^d} u_0 \,dx =1$. 

The main goal in this work is to give a precise meaning to the above formal reasoning, by obtaining an entropy solution to \eqref{010}--\eqref{012} (and thus \eqref{15}) defined in an appropriate sense.
For that, we will introduce a nonlocal parabolic equation containing a penalization term to enforce the constraint $u\le \theta$ (as in \cite{Levi1}), and a new, nonlocal Lagrange multiplier term designed to enforce the mass constraint \eqref{012}. As we will see below, this is not trivial to achieve. The first problem which arises is the lack of global in time existence for a possible solution of the problem \eqref{010}--\eqref{012}. This is explained in more detail below, and is linked to the possibility that the support of the solution may find itself in a region where the integral of $\theta$ is too small. In this way, it is obviously impossible to satisfy both conditions \eqref{011} and \eqref{012} simultaneously.

One can say  that such a situation reflects a lack of compatibility between the solution and the obstacle function. This problem is solved by defining an appropriate notion of \emph{compatibility} between the obstacle $\theta$ and the initial datum $u_0$ (see Definition~\ref{COMP} below). Crucially, this notion of compatibility is sufficient to obtain global-in-time existence of an entropy solution to the problem \eqref{010}--\eqref{012}.



The uniqueness of solution is not established here. Nevertheless, we conjecture that a wellposedness 
property is valid. Note that in \cite{Levi1}, the uniqueness property is a delicate part of that paper, 
as is usual in the theory of hyperbolic conservation laws. The difficulty in reproducing usual uniqueness arguments (Kruzkov's doubling of variables) is mainly due to the fact that a solution to \eqref{010}--\eqref{012} actually consists of a pair $(u,\lambda)$ (see Definition~\ref{DEFSOL} below). Note, however, that in order to obtain our existence result, a careful and involved study of a nonlocal parabolic problem is necessary, requiring in particular new assumptions on the data and delicate estimates. However, our method does not give an explicit or clear dependence of $\lambda(t)$ with respect to $u$.  
For these reasons we chose to leave for future work the interesting question of wellposedness.

Finally, it would be interesting to determine whether the methods in our paper can be extended to deal with more general (e.g., time dependent) mass constraints, hyperbolic systems 
of conservation laws, etc.
Also, it would be of great interest for physical applications (even under smoothness assumptions, to keep the analysis less involved) to extend the results of this work to a general conservation law 
with space and time dependent flux function and source term.

An outline of the paper follows.
In Section~\ref{S3000} we analyze the nonlocal parabolic problem which approximates the full problem \eqref{010}--\eqref{012}. The analysis is based on a fixed point argument. We present some details, since in the a priori estimates one must be careful due to the presence of the penalization and, especially, the nonlocal term.
Next, in Section~\ref{S4000} we provide key uniform estimates for the approximate problem.
They will allow us in Section~\ref{S5000} to pass to the limit on the penalized nonlocal parabolic equation to obtain a solution of \eqref{010}--\eqref{012}. Finally, in the Appendix, we provide a proof of two crucial lemmas, used to prove the uniform estimates of Section~\ref{S4000}.

\subsection{Smoothness assumptions on the data}

The initial datum $u_0$ is taken in the space $(L^\infty\cap L^1)(\RR^d).$
In fact, to simplify the exposition, we also consider that $u_0$ has bounded variation, that is, $u_0 \in BV(\RR^d)$. We suppose that the initial datum has unit mass, so
\be
\label{unit}
\aligned
\int_{\RR^d} u_0(x) \,dx =1.
\endaligned
\ee

The flux function $f$ is taken in $(C^2(\RR))^d$, and
without loss of generality we assume that  $f(0) =0$. Also, we suppose that
\be
\label{MM}
\aligned
\| f' \|_{(L^\infty(\RR))^d} \le M, \qquad  \| f'' \|_{(L^\infty(\RR))^d}  \le M'.
\endaligned
\ee
Note that in the subsequent analysis we will eventually prove an $L^\infty$ bound (for each time $T$) on the solution $u$, Therefore, the condition \eqref{MM} can then be relaxed in a standard way to
\[
\aligned
\sup_{|v| \le L }| f'(v) | \le M, \qquad  \sup_{|v| \le L }| f''(v) | \le M',
\endaligned
\]
where $M, M'$ may depend on $L$.

The obstacle $\theta(t,x): [0,+\infty) \times \RR^d \to \RR$ is assumed to satisfy the following conditions: 
\be
\label{50.2}
\aligned
\text{There exists a constant $\underline{\theta} >0$ such that }
 \theta(t,x) \ge \underline{\theta}, \qquad a.e. \quad(t,x),
\endaligned
\ee
\be
\label{50}
\aligned
\| \nabla_{t,x}  \theta \|_{ W^{2,1}([0,+\infty) \times \RR^d)} \le C_\theta,
\endaligned
\ee
\be
\label{50.1}
\aligned
\text{and for each compact set $K$, the function } t\mapsto \int_{K} \theta(t,x) \,dx \quad 
\text{is continuous.} 
\endaligned
\ee
Note that the obstacle $\theta(t,x)$ is not required to be bounded or continuous, but only bounded away from zero. Also, \eqref{50} is a condition on derivatives of order 1, 2 and 3 of $\theta$, but not on the function $\theta$ itself, which is not integrable.

\subsection{Entropy solutions to the obstacle-mass constraint problem}
Here we recall some standard facts and terminology from hyperbolic conservation laws. We refer the reader to the books \cite{GR} and \cite{Dafermos} for further reference on hyperbolic conservation laws.
\begin{definition}
A function $\eta \in C^1(\RR)$ is called an entropy for equation \eqref{010}, with associated entropy 
flux $q \in C^1(\RR; \RR^d)$, when for each $u \in \RR$, 
\begin{equation}
  q'_j(u)= \eta'(u) f'_j(u), \qquad (j= 1, \ldots, d).
\end{equation}
Also, we call $F(u)= (\eta(u), q(u))$ an entropy pair, and if $\eta$ is convex we say that $F(u)$ is a convex entropy 
pair. Moreover, $F(u)$  is called a generalized entropy pair if it is the uniform limit of a family of entropy pairs 
over compact sets. 
\end{definition}

The Kruzkov entropies are the most important example of generalized convex entropy pairs, consisting of the following parametrized family 
\be
\label{Kr}
   F(u,v)= ( |u-v|, \sgn (u-v) (f(u) - f(v)), \qquad v \in \RR. 
\ee
Next, extending the definition in \cite{Levi1}, we present in which 
sense a function $u(t,x)$ is a weak entropy solution of \eqref{010}--\eqref{012}.
 
\begin{definition}
\label{DEFSOL} Let $\theta: [0,+\infty) \times \RR^d \to \RR$ be a function which is called the obstacle, verifying the conditions in \eqref{50}--\eqref{50.1}. 
Let $u_0 \in (L^1 \cap L^\infty \cap BV)(\RR^d)$ with $0\le u_0(x) \le \theta(0,x)$  a.e., and $\int_{\RR^d} u_0 \, dx = 1$.
A pair $(u,\lambda)$ is called an obstacle mass conserving weak entropy solution of the Cauchy problem \eqref{010}--\eqref{012} if for any $T>0$:

$(i)$ The function $u$ is in $L^\infty  ((0,T) \times \RR^d)$ with $u(t,\cdot) \in BV(\RR^d)$ for a.a. $t \in [0,T]$, and the Lagrange multiplier $\lambda$ is in $L^\infty(0,T; \RR^+).$

\smallskip 
$(ii)$ For each nonnegative test function $\varphi \in C^\infty_c((-\infty,T) \times \RR^d)$, and any $k \in [0,1]$
\begin{equation}
\label{CL2}
\begin{aligned}
  \int_0^T\!\!  \int_{\RR^d}  &F(u(t,x), k \, \theta(t,x)) \cdot \nabla_{t,x} \varphi(t,x) \, dxdt
  \\[5pt] 
  &+ \int_0^T\!\! \int_{\RR^d}  \Big( \lambda(t) \, u(t,x) - H(k \, \theta(t,x)) \Big) 
  \sgn(u(t,x) - k \, \theta(t,x)) \, \varphi(t,x) \, dxdt
  \\[5pt]
  &+ \int_{\RR^d} | u_0(x) - k \,  \theta(0,x) | \, \vp(0,x) \, dx \geq 0.
\end{aligned}
\end{equation} 

 \smallskip
 $(iii)$ For almost all $(t,x)  \in (0,T) \times  \RR^d,$ $ \int_{\RR^d} u(t) \, dx= 1$ and $u(t,x) \le \theta(t,x)$. 
  
 \end{definition}

One observes that, as a consequence of Definition~\ref{DEFSOL}, the initial condition is assumed in the $L^1(\RR^d)$ strong sense (see \cite{ChenRascle,Panov1,Panov2}): 
 \begin{equation}
 \label{IC2}
    {\rm ess}\lim_{t \to 0} \int_{\RR^d} | u(t,x) - u_0(x) | \, dx= 0.
 \end{equation}

\subsection{Necessary and sufficient conditions for global-in-time existence: a compatibility condition}
\label{SS}

According to Definition~\ref{DEFSOL}, a solution $u$ to the obstacle-mass constraint problem must satisfy $0\le u \le \theta$ for almost all $t,x$. Then, it is obvious that if for some $t>0$ we have 
\[
\aligned
\int_{\{u(t) > 0\}} \theta (t,x) \,dx  < 1,
\endaligned
\]
then
\[
\aligned
\int_{\RR^d} u(t,x) \,dx = \int_{\{u(t) > 0\}} u(t,x) \,dx \le \int_{\{u(t) > 0\}} \theta (t,x) \,dx  < 1,
\endaligned
\]
which is in contradiction to the unit mass property (the last point in Definition~\ref{DEFSOL}). 

Thus, we see that a necessary condition for global-in-time existence is that 
the integral of the obstacle function $\theta$ on the support of the solution $u$ remain greater than one, 
that is to say, for almost all $(t,x)  \in (0,T) \times  \RR^d,$
\be
\label{1.000}
\aligned
\int_{\{u(t) > 0\}} \theta(t,x) \,dx \ge 1.
\endaligned
\ee


Of course the property \eqref{1.000} depends on the solution itself. To find a sufficient condition for global existence, we must find a condition on the initial datum $u_0$ and on the obstacle $\theta$ only, ensuring that a property like \eqref{1.000} remains valid for arbitrary times $T>0$. To this end, we now introduce the notion of \emph{compatible} initial datum and obstacle.

Define $v(t,x)$ as the unique entropy solution to the Cauchy problem for the conservation law on $[0,T] \times \RR^d$ 
\be
\label{403}
\aligned
&\del_t v + \dive f(v) =0,
\\
& v(0,x) = v_{0}(x),
\endaligned
\ee
where $v_0 \in (L^\infty\cap L^1\cap BV)(\RR^d)$. Recall from \cite{GR} that $v \in C(0,T; L^1(\RR^d))$.

\begin{definition}
\label{COMP}
Let $u_0$ be an initial datum and $\theta$ an obstacle verifying the assumptions in Definition~\ref{DEFSOL}. We say that $u_0$ and $\theta$ are \textbf{compatible} if there exists a function $v_0(x)  \in (L^\infty\cap L^1\cap BV)(\RR^d)$  
with the following properties:
\benum
\item[(i)]
 $v_0(x) \le \min( u_0(x), \underline{\theta})$, where $\underline{\theta}$ is the lower bound on the obstacle, given in \eqref{50.2};

\item[(ii)]
 For some $\beta >0$,
\be
\label{1.200}
\aligned
1+ \beta \le \int_{\{ v>0 \}} \theta (t,x) \,dx \le +\infty,
\endaligned
\ee
where $v(t,x)$ is given by \eqref{403}.
\eenum
\end{definition} 

We now show that, there exists an important special case, where one may ensure that $u_0$ and $\theta$ are compatible in the sense of Definition~\ref{COMP}.
\begin{proposition}
Suppose that for each compact $K \subset \RR^d$ there is a constant $c_K>0$ such that $ u_0 (x) \ge c_K$, $ x \in K$.
Then, $u_0$ and $\theta$ are compatible in the sense of Definition~\ref{COMP}.
\end{proposition}
\proof
It suffices to take some $0 < \gamma < \underline \theta$ and $v_0(x):= \min(u_0(x), \gamma)$ in \eqref{403}.
In that case, the condition \eqref{1.200} is valid.
Indeed, from finite speed of propagation, the solution $v$ of the conservation law \eqref{403} will have the same property of being locally bounded away from zero as $u_0$. To see this, consider a ball $B(r)$ of radius $r>0$ centered around an arbitrary point of $\RR^d$. Then we have that, for $M$ given by \eqref{MM}, 
%
%
$t>0$, the solution $v(t,x)$ on $B(r)$ 
is influenced only by the values of $v_{0}$ on $B(r+Mt)$. Let $c>0$ be such that $v_{0}\ge c$ on $B(r+(M+1)t)$.
Since $c$ is a solution to the conservation law \eqref{403}, the classical comparison property for conservation laws and domain of dependence arguments \cite{GR} imply that, $v(t,x) \ge c>0$ on $B(r).$ 

Therefore, $\{ v (t) > 0 \} = \RR^d$ and so the condition \eqref{1.200} in Definition~\ref{COMP} is verified.
\endproof


Note that Definition~\ref{COMP} refers to properties 
of the initial datum $u_0$, and the 
obstacle function $\theta$ only, and not of the solution to the obstacle-mass constraint problem $u(t,x)$. Indeed, it states that the support of the 
solution $v$ of the conservation law \eqref{403} cannot be carried into a region where the integral of $\theta$ is less than one.



\begin{remark}
\label{2000}
Suppose that $u_0$ and $\theta$ are compatible in the sense of Definition~\ref{COMP}. Then,
it follows that the initial datum $u_0$ must have some mass strictly below the obstacle $\theta$, which will be useful later. Indeed, suppose not, hence $u_0(x) = 0$ or $u_0(x) = \theta$. 
Therefore, $\{ v_{0} >0 \} \subset {\{u_0>0\}}$ and so, we would have
\[
\aligned
1< 1+ \beta \le \int_{\{ v_{0} >0\}} \theta \,dx \le \int_{\{ u_{0} >0\}} \theta \,dx = \int_{\RR^d} u_0 \,dx =1,
\endaligned
\]
which is a contradiction.
\end{remark}

\begin{remark}
\label{R2200}
Suppose we have a solution $u(x,t)$ of the obstacle-mass problem (Definition~\ref{DEFSOL}) with $\int_{\RR^d} u \,dx =1$, and that $u_0$ and $\theta$ are compatible in the sense of Definition~\ref{COMP}.
Then, the argument in Remark \ref{2000} is still valid for each $t>0$, showing that $u$ has some mass strictly below the obstacle $\theta$ for each $t$:
\be
\label{403.5}
\aligned
\int_{\{ u  < \theta\}} u(t)  \,dx \ge 0.
\endaligned
\ee
This property is crucial in our analysis. It guarantees that, if $u$ is losing mass from contact with the obstacle $\theta$, then there is a ``reserve'' of mass strictly below the obstacle on which the Lagrange multiplier term $u\lambda$ can act, compensating for the lost mass. One important part in this work is to prove rigorously a precise version of \eqref{403.5}, which can be found in Lemma~\ref{3700} below. The proof is delicate and can be found in the Appendix.

\end{remark}

For future reference, 
we will also consider the following viscous perturbation of \eqref{403}, 
\be
\label{404}
\aligned
&\del_t v_{\ve} + \dive f(v_{\ve}) -\ve \Delta v_{\ve} = 0,
\\
& v_{\ve}(0,x) = v_{0}(x).
\endaligned
\ee
The existence, uniqueness and regularity properties of the family
$\{v_{\ve}\}$, follow from standard well-posedness theory for
parabolic equations.

%

\subsection{Main result} 
 
The main result of this paper is the following existence theorem, which states that compatibility in the sense of Definition~\ref{COMP} is sufficient to ensure global-in-time existence of a solution to the obstacle-mass constraint problem.

\begin{theorem}[Existence of solution to the obstacle-mass constraint problem]
Let $u_0 \in (L^1\cap L^\infty \cap BV) (\RR^d)$, and let $\theta(t,x)$ be an obstacle function. Suppose that $u_0$, the flux $f$, and $\theta$ verify \eqref{unit}--\eqref{50.1} and that $u_0$, $\theta$ are compatible in the sense of Definition~\ref{COMP}.
Then, there exists an entropy solution to the hyperbolic obstacle-mass constraint problem \eqref{010}--\eqref{012} in the sense of Definition~\ref{DEFSOL}.
\label{TEO}
\end{theorem}

The proof of Theorem \ref{TEO} will be given in Section~\ref{S5000}. Our strategy consists of analyzing a perturbed problem (\eqref{315} below) and passing to the limit on the perturbation parameters. This analysis will be the object of the next sections.


\section{Study of the nonlocal penalized problem}
\label{S3000}

\subsection{An approach using a nonlocal penalization}

Let  $u_0$ and $\theta$ verify the assumptions \eqref{unit}--\eqref{50.1}. For each $\ve> 0$, and all $n \in \NN$ we consider the following nonlocal perturbed parabolic problem
\be
\label{315}
\left\{
\aligned
&\del_t u_{n,\ve} + \dive f(u_{n,\ve}) - \ve \Delta u_{n,\ve} = n\, u_{n,\ve}\, \int_{\RR^d} (u_{n,\ve} - \theta)^+ \, dx
- n (u_{n,\ve} - \theta)^+,
\\
&u_{n,\ve}(0,x) = u_0(x),
\endaligned
\right.
\ee
as an approximation scheme to solve the problem \eqref{010}--\eqref{012} (here $(z)^+= \max \{z,0\}$). 
Indeed, the last term in \eqref{315} is the usual term penalizing the excess of $u_{n,\ve}$ above $\theta$ (see \cite{Levi1}), ensuring that the limit of $u_{n,\ve}$ will stay below the obstacle $\theta$.

We now explain formally how the introduction in \eqref{315} of the nonlocal penalization term $n\, u_{n,\ve}\, \int_{\RR^d} (u_{n,\ve} - \theta)^+ \, dx$ implies the unit integral property. Indeed, integrating \eqref{315} on $\RR^d$ and on $[0,t]$, one finds using $\int_{\RR^d} u_0 \,dx =1$ that
\[
\aligned
\int_{\RR^d} u_{n,\ve}  \,dx - 1  &\le
n \int_0^t \Big( \int_{\RR^d}  u_{n,\ve}  \,dx -1 \Big) \Big( \int_{\RR^d} (u_{n,\ve} - \theta)^+ \,dx \Big) \,ds
\\
& \le n \sup_{(0,t)}\Big( \int_{\RR^d} (u_{n,\ve} - \theta)^+ \,dx \Big) \int_0^t \Big( \int_{\RR^d}  u_{n,\ve}  \,dx -1 \Big)  \,ds.
\endaligned
\]
Since $ u_{n,\ve}$ is expected to remain below the obstacle $\theta$ in the limit, the term $n\int_{\RR^d} (u_{n,\ve} - \theta)^+ \,dx $ is expected to remain bounded with $n$. Then,
using Gronwall's Lemma, the previous estimate yields $\int_{\RR^d} u_{n,\ve}(t,x) \,dx =1$ for $t\ge0$. The previous computation will be precisely described 
below. 

\subsection{{Well-posedness for the nonlocal penalized problem }}

In this section, we establish well-posedness results for the nonlocal penalized parabolic 
problem \eqref{315}. As we shall see, the analysis of this problem for each $n$ and $\ve$ is not trivial, 
due to the competition between the nonlocal term and the penalization term.
The main technical tool will be the Banach contraction principle. We follow in general lines the exposition in \cite{GR}. 

For $T>0$, define the space $\WW = \WW(0,T)$ by
\be
\label{WW}
\WW := \{v : v \in L^2 \big(0,T; H^1(\RR^d) \big), \del_t v \in 
L^2 \big(0,T; H^{-1}(\RR^d) \big) \}.
\ee
One recalls that the space $\WW$ enjoys the continuous imbedding 
$$
\WW \subset C \big([0,T]; L^2(\RR^d) \big).
$$
Moreover, for any $v \in \WW$ the $\lim_{t \to 0} v(t)= v(0)$ is a well defined element of the space $L^2(\RR^d)$.

\begin{theorem}[Well-posedness for the nonlocal penalized problem]
\label{3500}
Let $u_0 \in (L^1\cap L^2)(\RR^d)$, with $\int_{\RR^d} u_0(x) \,dx =1$. Then, for each $n \in \NN, \ve, T> 0$,
there exists a unique solution 
$$
   u_{n,\ve} \in \WW  \cap C \big([0,T]; (L^1\cap L^2)(\RR^d) \big),
$$   
of the nonlocal 
parabolic problem \eqref{315}, in the sense that: For every $w\in H^1(\RR^d)$, and for almost all $t\in (0,T)$,
\be
\label{318}
\aligned
&\langle \del_t u_{n,\ve}(t), w \rangle_{H^{-1}\times H^1} -
\int_{\RR^d} \big(f(u_{n,\ve}(t)) - \ve\nabla u_{n,\ve}(t)\big) \cdot \nabla w \,dx 
\\
&\qquad \qquad
= n \Big( \int_{\RR^d}  u_{n,\ve}(t) \, w \,dx\Big) \Big( \int_{\RR^d} (u_{n,\ve}(t) - \theta(t))^+  \,dx \Big)
\\
&\qquad \qquad \quad- n \int_{\RR^d} (u_{n,\ve}(t) - \theta(t))^+\,w \,dx,
\endaligned
\ee 
and $\lim_{t\to 0}\int_{\RR^d} \| u_{n,\ve}(t) - u_0\|_{L^2(\RR^d)} dx \to 0$.
Moreover, this solution verifies for almost all $t \in (0,T)$, 
$$\int_{\RR^d} u_{n,\ve}(t) dx = 1.$$
\end{theorem}

\begin{proof}
1. The theorem will be proved using the Banach contraction principle. 
To this end, given $\overline v \in C\big([0,T]; L^1(\RR^d)\big)$, let $v \in \WW$ be the weak solution of the Cauchy problem
\be
\label{332}
\aligned
&\del_t v + \dive f(v) - \ve \Delta v = n\, v\, \int_{\RR^d} (\overline v - \theta)^+ \,dx - n (v - \theta)^+,
\\
&v(0,x) = u_0(x).
\endaligned
\ee
More precisely, $v$ is such that, for all $w\in H^1(\RR^d)$, and for almost all $t\in (0,T)$,
\be
\label{330}
\aligned
&\langle \del_t v(t), w \rangle_{H^{-1}\times H^1} -
\int_{\RR^d} \big(f(v(t)) - \ve \nabla v(t)\big) \cdot \nabla w \,dx 
\\
&\qquad \qquad
= n \Big(\int_{\RR^d} v(t) \, w \,dx \Big)\Big(\int_{\RR^d} (\overline v(t)- \theta)^+  \,dx \Big) - n \int_{\RR^d} (v(t)- \theta)^+\,w \,dx. 
\endaligned
\ee 
Moreover, $\lim_{t\to 0}\int_{\RR^d} \| v(t) - u_0\|_{L^2(\RR^d)} dx \to 0$. 
The proof that there exists a unique solution of \eqref{330} follows closely the one in \cite[p.56]{GR}, 
so we omit it. Note that \eqref{330} is a standard (local) parabolic problem.

\smallskip
2. Now, let us consider the mapping
\be
\label{335}
\aligned
\Phi : {}& C\big([0,T]; L^1(\RR^d)\big) \to \WW,
\\
\overline v &\mapsto \Phi(\overline v) = v \text{ solution of }\eqref{330}.
\endaligned
\ee 
Let $R>1$ to be chosen later. We will show that, for $T_0$ sufficiently small, $\Phi$ is a contraction in the Banach space
\be
\label{340}
\mathcal{E} := \{ v \in C\big([0,T_0]; L^1(\RR^d)\big) :
\| v \|_\mathcal{E} :=
\sup_{t\in [0,T_0]}  \| v(t) \|_{L^1(\RR^d)} \le R \}.
\ee
Let $v$ be the unique solution of problem \eqref{330}.
First of all, note that since $u_0 \ge 0$, we have $v\ge0$. This follows from the fact that $v\equiv 0$ is a solution of the problem \eqref{330} and classical comparison arguments 
(see, in particular, Lemma~\ref{420} below).

3. We prove that the map $\Phi$ takes $\Ecal$ into $\Ecal$.
For this, we establish the estimate
\be
\label{350}
\aligned
\int_{\RR^d} v(t) \,dx \le e^{n\int_0^t \int_{\RR^d}(\overline v - \theta)^+ \,dx \,ds }.
\endaligned
\ee
Note that once \eqref{350} is proved, we find immediately
\[
\aligned
\int_{\RR^d} v(t) \,dx 
& \le e^{n t  \| \overline v  \|_\Ecal  }  \le e^{nt R}.
\endaligned
\]
Now, since $R>1$, we have that for $t \le T_0$ sufficiently small, $e^{nt R} < R$ and so $v \in \Ecal$.

To prove \eqref{350}, we introduce as in \cite[p.54]{GR}, and for the sake of localization, the smooth positive functions $\psi_\rho : \RR^d \to \RR$ for large $\rho \in \RR$, such that for some constant $C>0$,
\be
\label{352}
\aligned
&\psi_\rho(x) =1 \text{ if } |x| \le \rho/2, \quad \psi_\rho \text{ 
decays exponentially for } |x| \ge \rho,
\\
&|\nabla \psi_\rho| \le  \frac{C\psi_\rho}\rho, \quad \text{ and }
\quad |\Delta \psi_\rho| \le  \frac{C\psi_\rho}{\rho^2}.
\endaligned
\ee
Take $w = \psi_\rho$ as a test function in \eqref{330} to find, after discarding the 
(nonpositive) last term on the right-hand side,
\be
\label{INTERM}
\aligned
&\langle \del_t v(t), \psi_\rho \rangle_{H^{-1}\times H^1} -
\int_{\RR^d} \big( f(v(t)) - \ve \nabla v(t) \big) \cdot \nabla \psi_\rho \,dx 
\\
&\qquad \qquad
\le n \Big( \int_{\RR^d} v(t) \, \psi_\rho \,dx \Big) \Big( \int_{\RR^d} (\overline v(t)- \theta)^+ \,dx \Big).
\endaligned
\ee
Now, for $\rho$ sufficiently large, recalling $ | f(v) | = |f(v) - f(0)| \le Mv$ and \eqref{352},
\[
\aligned
\int_{\RR^d} f(v) \cdot \nabla \psi_\rho - \ve \nabla v \cdot \nabla \psi_\rho \,dx 
&\le \int_{\RR^d}M v |\nabla \psi_\rho| + \ve \,v |\Delta \psi_\rho|  \,dx 
\\
& \le \int_{\RR^d} C \big(M \frac{v}\rho \psi_\rho + \ve \,\frac{v}{\rho^2} \psi_\rho \big) \,dx 
\\
&\le \frac {C(M+\ve )}\rho \int_{\RR^d} v \psi_\rho \,dx.
\endaligned
\]
Hence integrating \eqref{INTERM} on $(0,t)$, we find
\[
\aligned
\int_{\RR^d} v(t) \psi_\rho \,dx &\le 
\int_{\RR^d} u_0 \psi_\rho \,dx 
\\
& \quad + \int_0^t \Big( C\frac{M+\ve }\rho + n \int_{\RR^d} (\overline v(s) - \theta)^+ \,dx \Big) \Big( \int_{\RR^d} v(s) \psi_\rho \,dx  \Big)\,ds.
\endaligned
\]
Using Gronwall's inequality and $\int_{\RR^d} u_0 \psi_\rho \,dx \le \int_{\RR^d} u_0 \,dx = 1$, we obtain
\[
\aligned
\int_{\RR^d} v(t) \psi_\rho \,dx \le e^{Ct\frac{M+\ve }\rho} 
e^{n \int_0^t \int_{\RR^d} (\overline v - \theta)^+ \,dx \,ds}.
\endaligned
\]
Then, passing to the limit as $\rho \to \infty$ and applying the Monotone Convergence Theorem, we conclude that 
\[
\aligned
\int_{\RR^d} v(t) \,dx \le 
e^{n \int_0^t \int_{\RR^d} (\overline v - \theta)^+ \,dx \,ds},
\endaligned
\]
which is \eqref{350}. This proves that the map $\Phi$ takes $\Ecal$ into $\Ecal$ (recall the comments after \eqref{350}).

\smallskip
3. We will now prove that the map $\Phi$ is a contraction on $\Ecal$ for sufficiently small $T_0$. That is, we will prove the following estimate,
\be
\label{390}
\aligned
\| u - v \|_\Ecal   &\le K \| \overline u - \overline v \|_{\Ecal},
\endaligned
\ee
for some $K <1$.

Let us fix some notations.
We introduce for $\delta > 0,$ $u \in \RR$, the regularized sign function $\sgn_\delta(u)$ as the continuous function which is linear for $0\le |u|\le\delta$, and equal $\pm 1$ otherwise. Also, we use the notations
\be
\label{sign1}
\aligned
I_\delta(u) = \int_0^u \sgn_\delta(v)^+ \,dv,\qquad u\in \RR, 
\endaligned
\ee
and
\be
\label{sign2}
\aligned
(u)^+_\delta = u \, \sgn_\delta(u)^+,
\endaligned
\ee
both of which are Lipschitz approximations of the positive part $u^+$. Note that a small calculation gives
\be
\label{sign3}
\aligned
I_\delta (u) \le (u)_\delta^+ \le u^+, \qquad \delta, u \ge 0.
\endaligned
\ee

Next, let $\overline u, \overline v \in \Ecal$ and let $u$ and $v$ be solutions of \eqref{330} associated with $\overline u$ and $\overline v$, respectively.
Recall the definition of the function $\psi_\rho$ in \eqref{352}.
Then, 
write \eqref{330} for $u$ and $v$, and subtract.
Taking $w = \sgn_\delta(u-v)^+ \psi_\rho$ as a test function, we find, with obvious notation,
\be
\label{370}
\aligned
\langle &\del_t I_\delta(u - v), \psi_\rho \rangle_{H^{-1}\times H^1} 
\\
&=
\int_{\RR^d} \big( f(u) - f(v) - \ve \nabla (u-v) \big) \cdot \nabla \big( \sgn_\delta(u- v)^+ \psi_\rho \big) \,dx  
\\
&\quad 
+ n \Big (\int_{\RR^d} u\, \sgn(u-v)^+_\delta \psi_\rho \,dx \Big) \Big( \int_{\RR^d} (\overline u- \theta)^+ \,dx \Big)
\\
&\quad- n \Big(\int_{\RR^d} v\, \sgn(u-v)^+_\delta \psi_\rho \,dx \Big)  \Big(\int_{\RR^d} (\overline v- \theta)^+ \,dx \Big)
\\
&\quad- n \int_{\RR^d} \big( (u- \theta)^+ - (v- \theta)^+ \big) \sgn_\delta(u-v)^+ \psi_\rho \,dx 
\\
& = I_1 + I_2 + I_3 + I_4.
\endaligned
\ee
Now, observe that
for each $\delta\ge0,$
the algebraic inequality holds,
\[
\aligned
\big( (u- \theta)^+ - (v- \theta)^+ \big) \sgn_\delta(u-v)^+ \ge 0,
\endaligned
\]
as is easily seen by considering the various cases.
This shows that
\be
\label{I4}
\aligned
I_4 \le 0.
\endaligned
\ee

We now treat the term $I_1$ in \eqref{370}. We find
\be
\label{I1234}
\aligned
I_1& = \int_{\RR^d} (f(u) - f(v) - \ve \nabla (u-v)) \cdot \nabla (\sgn_\delta(u- v)^+ \psi_\rho) \,dx 
\\
&=  \int_{\RR^d} (f(u) - f(v)) \cdot \nabla (u-v)  \sgn_\delta'(u- v)^+ \psi_\rho \,dx  
\\
&\quad +  \int_{\RR^d} (f(u) - f(v)) \cdot  \sgn_\delta(u-v)^+ \nabla \psi_\rho \,dx
\\
& \quad - \ve \int_{\RR^d} | \nabla (u-v)|^2 \sgn_\delta'(u-v)^+ \psi_\rho \,dx
\\
&\quad- \ve  \int_{\RR^d}  \nabla (u-v) \sgn_\delta(u-v)^+ \nabla \psi_\rho \,dx
\\
& = I_{11} + I_{12} - I_{13} - I_{14}.
\endaligned
\ee
First, using \eqref{MM} and $ab \le \frac\ve{2} a^2 + \frac{1}{2\ve} b^2$, we get
\[
\aligned
I_{11} &\le M \int_{\RR^d} (u - v)^+ | \nabla (u - v) |  \sgn_\delta'(u-v) \psi_\rho \,dx
\\
& \le \frac{M^2}{2\ve} \int_{\RR^d} ((u - v)^+)^2   \sgn_\delta'(u-v) \psi_\rho \,dx 
\\
&\qquad + \frac\ve2   \int_{\RR^d}   | \nabla (u - v) |^2  \sgn_\delta'(u-v) \psi_\rho \,dx.
\endaligned
\]
In this way,
\be
\label{I131}
\aligned
I_{11} - I_{13} 
& \le \frac{M^2}{2\ve} \int_{\RR^d} ((u - v)^+)^2   \sgn_\delta'(u-v) \psi_\rho \,dx 
\\
&\qquad - \frac\ve2  \int_{\RR^d}   | \nabla (u - v) |^2  \sgn_\delta'(u-v) \psi_\rho \,dx
\\
& \le \frac{M^2}{2\ve} \int_{\RR^d} ((u - v)^+)^2  \sgn_\delta'(u-v) \psi_\rho \,dx .
\endaligned
\ee
Let us introduce a convenient notation: we denote by 
\[
\aligned
o(\delta,t) 
\endaligned
\]
a function of $\delta\in \RR$ which tends to zero when $\delta \to 0$ for almost every $t>0$.
Note that 
$$\sgn_\delta' (u)^+ = \frac{1}{\delta} \chi_{\{0<u<\delta\}},$$
where here and in what follows $\chi_A$ denotes the carachteristic function of a set $A \subset \RR^d$.

Now, for each $\rho,$ the family of functions of $(t,x)$ appearing on the right-hand side of \eqref{I131}, 
\[
\zeta_\delta (t,x)  := \big( (u(t,x) - v(t,x))^+\big)^2 \sgn_\delta'(u(t,x)-v(t,x))^+ \psi_\rho 
\]
satisfies for almost every $(t,x)\in [0,\infty) \times \RR^d$
\[
\aligned
0 \le \zeta_\delta (t,x) & = ((u - v)^+)^2  \frac1\delta \chi_{\{ (u-v)^+ \le \delta \}} \psi_\rho
\\
& \le (u- v)^+  \chi_{\{ (u-v)^+ \le \delta \}} \psi_\rho
\\
&\le \delta \psi_\rho \in L^1(\RR^d)
\endaligned
\]
and
\[
\aligned
\zeta_\delta (t,x)  \to 0, \quad \delta \to 0.
\endaligned
\]
These facts and Lebesgue's Theorem show that (cf. \eqref{I131}), 
\be
\label{I13}
\aligned
\int_0^t I_{11}(s) - I_{13}(s) \,ds = \frac{M^2}{2\ve} \int_0^t \int_{\RR^d} \zeta_\delta (s,x) \,dx \,ds= o(\delta,t).
\endaligned
\ee

Going back to \eqref{I1234}, using the properties of $\psi_\rho$ in \eqref{352}, and also \eqref{sign3}, we find
\[
\aligned
I_{12} - I_{14} & \le M \int_{\RR^d} (u - v)^+ \sgn_\delta (u - v)^+ |\nabla \psi_\rho | \,dx
\\
& \quad - \ve \int_{\RR^d} \nabla ( I_\delta (u - v)) \cdot \nabla \psi_\rho \,dx
\\
& = M \int_{\RR^d} (u - v)^+_\delta |\nabla \psi_\rho | \,dx
 + \ve \int_{\RR^d}  I_\delta (u - v)) \Delta \psi_\rho \,dx
\\
& \le  \frac{CM}\rho\int_{\RR^d} (u - v)^+_\delta   \psi_\rho \,dx 
 + \frac{\ve C}{\rho^2} \int_{\RR^d} I_\delta(u - v) \psi_\rho \,dx
\\ 
& \le \frac{C(M+\ve )}\rho \int_{\RR^d} (u - v)^+ \psi_\rho \,dx.
\endaligned
\]
This estimate along with \eqref{I13} gives
\be
\label{I1}
\aligned
\int_0^t I_1(s) \,ds & \le  \mathop{o}(\delta) + \frac{C(M+\ve )}\rho \int_0^t \int_{\RR^d} (u - v)^+ \psi_\rho \,dx \,ds.
\endaligned
\ee
We now turn to the remaining terms in \eqref{370}.
\[
\aligned
I_2 + I_3
& =
 n \Big( \int_{\RR^d} u\, \sgn(u-v)^+_\delta \psi_\rho \,dx \Big) \Big( \int_{\RR^d} (\overline u- \theta)^+ \,dx \Big)
\\
&\quad- n \Big( \int_{\RR^d} v\, \sgn(u-v)^+_\delta \psi_\rho \,dx \Big) \Big( \int_{\RR^d} (\overline v- \theta)^+ \,dx \Big)
\\
& \le n \Big( \int_{\RR^d} (u - v)^+_\delta \psi_\rho \,dx \Big)  \Big( \int_{\RR^d} (\overline u- \theta)^+ \,dx \Big)
\\
&\qquad+ n \Big( \int_{\RR^d} v\, \sgn_\delta(u-v)^+ \psi_\rho \,dx \Big) \Big( \int_{\RR^d}(\overline u- \theta)^+ - (\overline v- \theta)^+ \,dx \Big)
\\
& \le  n \Big( \int_{\RR^d} (u - v)^+ \psi_\rho \,dx \Big) \Big( \int_{\RR^d} (\overline u- \theta)^+ \,dx \Big)
\\
&\qquad+ n \int_{\RR^d} v \,dx \int_{\RR^d}(\overline u - \overline v)^+\,dx,
\endaligned
\]
and so
\be
\label{I23}
\aligned
\int_0^t I_2(s) + I_3(s) \,ds
& \le
n \int_0^t \Big( \int_{\RR^d} (u - v)^+ \psi_\rho \,dx \Big) \Big( \int_{\RR^d} (\overline u- \theta)^+ \,dx \Big) \,ds
\\
&\qquad+ n \int_0^t \Big( \int_{\RR^d} v \,dx \Big) \Big( \int_{\RR^d}(\overline u - \overline v)^+\,dx \Big) \,ds.
\endaligned
\ee
Therefore, integrating \eqref{370} on $(0,t)$ yields from \eqref{I4}, \eqref{I1} and \eqref{I23},
\[
\aligned
\int_{\RR^d} & I_\delta(u(t) - v(t)) \psi_\rho   \,dx  \le 
\mathop{o}(\delta) + \int_0^t \frac{C(M+\ve )}\rho \int_{\RR^d} (u - v)^+ \psi_\rho \,dx \,ds
\\
&+ n \int_0^t \Big( \int_{\RR^d} (u - v)^+ \psi_\rho \,dx \Big) \Big( \int_{\RR^d} (\overline u- \theta)^+ \,dx \Big) \,ds
\\
&+ n \int_0^t \Big( \int_{\RR^d} v \,dx \Big) \Big( \int_{\RR^d}(\overline u - \overline v)^+\,dx \Big) \,ds.
\endaligned
\]
Now, we apply the Monotone Convergence Theorem to take $\delta \to 0$, and use Gronwall's Lemma to get 
\[
\aligned
\int_{\RR^d} (u (t)- v(t))^+ \psi_\rho   \,dx 
&\le 
n \int_0^t \Big( \int_{\RR^d} v \,dx\Big)\Big( \int_{\RR^d}(\overline u - \overline v)^+ \,dx \Big)\,ds 
\\
&\qquad \times
\exp \Big({n\int_0^t \int_{\RR^d} (\overline u - \theta)^+ \,dx \,ds +
t\frac{C(M+\ve)}\rho } \Big).
\endaligned
\]
Then, taking $\rho \to \infty$ (again by monotone convergence), we obtain the estimate
\be
\label{360}
\aligned
\int_{\RR^d} (u - v)^+ (t) \,dx &\le 
n \int_0^t \Big(\int_{\RR^d} v \,dx\Big)\Big( \int_{\RR^d} (\overline u - \overline v)^+ \,dx \Big) \,ds
\\
&\qquad \times \exp\Big( {n \int_0^t \int_{\RR^d} (\overline u - \theta)^+ \,dx \,ds} \Big).
\endaligned
\ee
Now we use the estimate \eqref{350} of $\int_{\RR^d} v \,dx$ in \eqref{360} to find
\[
\aligned
\int_{\RR^d} (u - v)^+(t) \,dx &\le n\int_0^t \int_{\RR^d} (\overline u - \overline v)^+ \,dx \,ds 
\\
&\qquad \times \exp\Big( {n\int_0^t\int_{\RR^d} (\overline u - \theta)^+ + (\overline v - \theta)^+ \,dx \,ds} \Big).
\endaligned
\]
or, recalling the definition of the space $\Ecal$ in \eqref{340},
\[
\aligned
\int_{\RR^d} (u - v)^+(t) \,dx &\le n t \sup_{0\le t \le T_0}\| \overline u - \overline v \|_{L^1(\RR^d)} \, e^{nt \sup_{0\le t \le T_0}(\| \overline u\|_{L^1(\RR^d)} + \| \overline v\|_{L^1(\RR^d)})}
\\
&\le  n t\| \overline u - \overline v \|_{\Ecal} \, e^{2R nt}.
\endaligned
\]
By symmetry, we find an estimate equal to the previous one, but with $(v - u)^+$ instead of $(u - v)^+.$ From $|a| = a^+ + (-a)^+$, we have
\[
\aligned
\int_{\RR^d} |u - v| (t) \,dx &\le 2 n t e^{2R nt} \| \overline u - \overline v \|_{\Ecal} ,
\endaligned
\]
and consequently, choosing $T_0$ such that $\sup_{t \in (0,T_0)} 2nt e^{2R nt} \le K < 1$,
\[
\aligned
\| u - v \|_\Ecal   &\le K \| \overline u - \overline v \|_{\Ecal} ,
\endaligned
\]
which is the desired contraction estimate \eqref{390}. This proves contraction of the map $\Phi$ for sufficiently small $T_0$.

\smallskip
4.
We are now in a position to finish the existence proof for the penalized viscous problem \eqref{318}.
The first part of the Banach Contraction Principle tells us that the sequence defined by $u^k = \Phi(u^{k-1})$ with $u^0\in \Ecal$ converges strongly in $\Ecal$ towards some $ u_{n,\ve} \in \Ecal$. Each $u^k$ verifies equation \eqref{330} with $u^{k-1}$ in place of $\overline v$. Taking $w = u^k$ in \eqref{330} gives
\[
\aligned
&\frac12 \int_{\RR^d} \del_t (u^k)^2 \,dx  +
\ve \int_{\RR^d} | \nabla u^k |^2 \,dx 
- \int_{\RR^d} f(u^k) \cdot \nabla u^k \,dx
\\
&\qquad \le n \int_{\RR^d} (u^k)^2 \,dx \Big( \int_{\RR^d} ( u^{k-1} - \theta)^+ \,dx \Big) 
\\
&\qquad\le n R \int_{\RR^d} (u^k)^2 \,dx,
\endaligned
\]
since $u^{k-1} \in \Ecal$. Furthermore, since
\[
\aligned
\int_{\RR^d} f(u^k) \cdot \nabla u^k \,dx \le \frac{M^2}{2\ve} \int_{\RR^d} (u^k)^2 \,dx + \frac\ve2 \int_{\RR^d} | \nabla u^k |^2 \,dx , 
\endaligned
\]
we get
\[
\aligned
&\frac12 \frac{d}{dt} \int_{\RR^d}  (u^k)^2 \,dx  +
\frac\ve2 \int_{\RR^d} | \nabla u^k |^2 \,dx 
\le \big( n R + \frac{M^2}{2\ve} \big)  \int_{\RR^d} (u^k)^2 \,dx.
\endaligned
\]
Integrating this inequality on $(0,t), t\le T_0$ and applying Gronwall's lemma gives first
\[
\aligned
u^k \in L^\infty(0,t; L^2(\RR^d))
\endaligned
\]
and then also
\[
\aligned
u^k \in L^2(0,t; H^1(\RR^d)), \quad t\le T_0,
\endaligned
\]
uniformly in $k$. This allows us to conclude that the limit $u_{n,\ve}$ (which is in $\Ecal$ by definition) is also in $\WW(0,t)$ (cf.~\eqref{WW}) and so solves the problem \eqref{318}, at least for some time $T_0$. Since functions in $\WW(0,T_0)$ are actually continuous on $[0,T_0]$ with values in $L^2(\RR^d)$ (see \cite[p.54]{GR}), the initial datum $u_0$ is indeed assumed. This completes the existence part of the proof of Theorem~\ref{3500}.

\smallskip
5. Finally, we show global in time existence. For this it will be sufficient to prove that $\int_{\RR^d} u_{n,\ve} (t) \,dx = 1$ for almost all $t\in [0,T_0)$. In \eqref{318} take $\psi_\rho$ as test function to obtain in a way very similar to what was used to deduce \eqref{I1}, using also the properties of $\psi_\rho$ in \eqref{352},
\[
\aligned
\frac{d}{dt} \langle u_{n,\ve}(t), \psi_\rho \rangle  &\le \frac{C(M+\ve )}\rho \int_{\RR^d} u_{n,\ve}(t) \, \psi_\rho \,dx 
\\
& \quad+ n \Big(  \int_{\RR^d} u_{n,\ve}(t) \, \psi_\rho \,dx \Big) \Big( \int_{\RR^d} (u_{n,\ve}(t) - \theta)^+ \,dx \Big) 
\\
&\quad - n \int_{\RR^d} (u_{n,\ve}(t) - \theta)^+ \psi_\rho \,dx
\endaligned
\]
and so, integrating on $(0,t)$,
\be
\label{395}
\aligned
\int_{\RR^d} u_{n,\ve}(t) \, \psi_\rho \,dx  &\le
\int_{\RR^d} u_0 \,dx +  \frac{C(M+\ve )}\rho \int_0^t\int_{\RR^d} u_{n,\ve}(s) \, \psi_\rho \,dx \,ds 
\\
&\quad+ n  \int_0^t \Big(\int_{\RR^d} u_{n,\ve}(s) \, \psi_\rho \,dx \Big) \Big( \int_{\RR^d} (u_{n,\ve}(s) - \theta)^+ \,dx \Big) \,ds
\\
&\quad - n \int_0^t \int_{\RR^d} (u_{n,\ve}(s) - \theta)^+ \psi_\rho \,dx \,ds.
\endaligned
\ee
Since $u_{n,\ve} \in \Ecal,$ we have $u_{n,\ve}(t) \in L^1(\RR^d)$. Now, we return to \eqref{395} and take $\rho\to\infty$ applying Lebesgue's Theorem to find
\[
\aligned
\int_{\RR^d} u_{n,\ve}  \,dx  &\le
\int_{\RR^d} u_0 \,dx 
+ n  \int_0^t \Big(\int_{\RR^d} u_{n,\ve} \,dx \Big) \Big( \int_{\RR^d} (u_{n,\ve} - \theta)^+ \,dx \Big) \,ds
\\
&\quad - n \int_0^t \int_{\RR^d} (u_{n,\ve} - \theta)^+  \,dx \,ds.
\endaligned
\]
Then, it follows that 
\[
\aligned
\int_{\RR^d} u_{n,\ve}  \,dx - 1  &\le
n \int_0^t \Big( \int_{\RR^d}  u_{n,\ve}  \,dx -1 \Big) \Big( \int_{\RR^d} (u_{n,\ve} - \theta)^+ \,dx \Big) \,ds
\\
&\le n R \int_0^t \Big( \int_{\RR^d}  u_{n,\ve}  \,dx -1 \Big) \,ds,
\endaligned
\]
and consequently, by Gronwall's lemma, $\int_{\RR^d} u_{n,\ve} (t) \,dx =1,$ for almost all $t\le T_0.$ Furthermore, since $u\in \Ecal$, we may suppose by continuity of the $L^1$ norm that
\be
\label{397}
\aligned
\forall\, t\le T_0, \qquad 
\int_{\RR^d} u_{n,\ve} (t) \,dx =1.
\endaligned
\ee
This completes the proof of Theorem~\ref{3500}.
\end{proof}

\section{Uniform estimates for the nonlocal penalized problem}
\label{S4000}

In this section, we prove estimates for solutions of \eqref{315} independently of the penalization parameter $n$ and of the viscosity parameter $\ve$. They will allow not only the necessary compactness properties on the sequence $(u_{n,\ve})$ but also give a more precise characterization of the limit of $u_{n,\ve}$ as $n\to \infty$, $\ve \to 0$. So, in Theorem~\ref{4000} we prove an estimate which ensures that, in the limit, the solution of the obstacle-mass constraint problem will indeed stay below the obstacle. For this, we need the result in Lemma~\ref{3700} (whose proof is found in the Appendix), which states that the solutions $u_{n,\ve}$ retain some mass strictly below the obstacle, uniformly in $n$. Recall from the discussion in Section~\ref{SS}, that the compatibility property in Definition~\ref{COMP} was especially designed to ensure this type of property.

\medskip
Then, in Theorem~\ref{6000}, we establish uniform (in $n$ and $\ve$) estimates for $u_{n,\ve}$ in $W^{1,1}((0,T)\times \RR^d)$. These estimates will allow us in the next section to obtain existence of a solution for problem \eqref{010}--\eqref{012}, using the vanishing viscosity method.

\subsection{Main estimates independent of $n$ and $\ve$}

\begin{theorem}
\label{4000}
Let $T>0$ be arbitrary. Suppose the initial datum $u_0$ is in $(L^1 \cap L^\infty \cap BV)(\RR^d)$ with $\int_{\RR^d} u_0 \,dx =1$ and that, $u_0$ and $\theta$ are compatible in the sense of Definition~\ref{COMP}. Let $\{u_{n,\ve}\}$ be the family of solutions of the nonlocal parabolic problem \eqref{318}. 
Then, there exist constants $\alpha >0,$ $C_\theta$ depending on $T$, $u_0$ and $\theta$, but not on $n$, such that for all $\ve> 0$ sufficiently small, and a.e. $t \in (0,T)$
\be
\label{400}
\aligned
\int_{\RR^d} \big( u_{n,\ve}(t) - \theta(t) \big)^+ \,dx  \le \frac {C_\theta}{\alpha n},
\endaligned
\ee
\be
\label{409}
\aligned
\| u_{n,\ve}(t) \|_{L^\infty(\RR^d)} \le  \| u_0 \|_{ L^\infty(\RR^d)}  e^{t\frac{C_\theta}\alpha}.
\endaligned
\ee
The constant $\alpha$ is given by Lemma~\ref{3700} below.
\end{theorem}

To prove Theorem~\ref{4000}, we consider the following key result, which was discussed in Remark~\ref{R2200}.
\begin{lemma}
\label{3700}
Under the same conditions of Theorem~\ref{4000}, there exists a constant $\alpha >0$ depending on $T$, $u_0$ and $\theta$, but not on $n$ or $\ve$, such that the estimate is valid:
\be
\label{410}
\aligned
\inf_{0\le t \le T} \int_{\{ u_{n,\ve}  < \theta\}} u_{n,\ve}(t)  \,dx \ge \alpha.
\endaligned
\ee
\end{lemma}

We also have
\begin{theorem}
\label{6000}
Under the same conditions of Theorem~\ref{4000},
the solution $ u_{n,\ve}$ of the nonlocal penalized parabolic problem \eqref{318} with regularized initial datum satisfies
\[
\aligned
u_{n,\ve} \in W^{1,1}((0,T)\times \RR^d), \qquad \text{uniformly in } \ve,n.
\endaligned
\]
More precisely, for each $n \in \NN$, and $\ve> 0$, and almost all $t \in (0,T)$
\[
\aligned
&\| \del_t u_{n,\ve}(t) \|_{L^1(\RR^d)} + \| \nabla u_{n,\ve}(t) \|_{L^1(\RR^d)} \le C \big(  \TV (u_0) + 1
 \big) e^{tC},
\endaligned
\]
where  $C$ depends on $\theta$, $f$, $d,$ $t$ and $\alpha$ (given by Lemma~\ref{3700}) but not on $\ve$ or $n$.
\end{theorem} 
Here, $\TV$ denotes the total variation, see for instance \cite[p.51]{GR}.

\begin{remark}
Let us comment briefly on the results of Lemma~\ref{3700} and Theorem~\ref{4000}. The estimate \eqref{410} states that, for $t\in [0,T]$, the function $u_{n,\ve}$ retains some mass below the obstacle $\theta$, \emph{uniformly in $n$ and $\ve$}, 
and it is the most delicate estimate in this work. The key compatibility property in Definition~\ref{COMP} is used to prove the estimate \eqref{410}, which in turn ensures the key property \eqref{400}. This last estimate ensures that as $n\to \infty$ the mass above the obstacle $\theta$ of the solutions $u_{n,\ve}$ vanishes.

Also, although a smoother initial datum is required for Theorem~\ref{6000}, when passing to the limit $n\to \infty, \ve\to 0$ this requirement can be eliminated in a completely standard way. We omit this straightforward procedure (found, e.g., in \cite{GR}) for the sake of clarity.
\end{remark}

Now we prove Theorems~\ref{4000} and \ref{6000}, leaving the proof of Lemma~\ref{3700} to the Appendix.

\begin{proof}[Proof of Theorem \ref{4000}]

\bigskip
We prove the estimate \eqref{400}. Recall the notations \eqref{sign1} and \eqref{sign2}. Use the weak formulation \eqref{318} with $w = \sgn_\delta(u_{n,\ve}-\theta)^+$ to find
after adding and subtracting various terms, 
\be
\label{445}
\aligned
&\frac d{dt} \int_{\RR^d} I_\delta(u_{n,\ve}-\theta)^+ \,dx - \int_{\RR^d} ( f(u_{n,\ve}) - f(\theta)) \cdot  \nabla \sgn_\delta(u_{n,\ve}-\theta)^+ \,dx  
\\
&\qquad+ \ve \int_{\RR^d} \sgn_\delta'(u_{n,\ve}-\theta)^+ |\nabla(u_{n,\ve}-\theta)|^2 \,dx 
\\
&\quad= n  \int_{\RR^d} u_{n,\ve} \sgn_\delta(u_{n,\ve}-\theta)^+ \,dx  \cdot \int_{\RR^d} (u_{n,\ve} - \theta)^+ \,dx 
\\
& \qquad- n \int_{\RR^d} (u_{n,\ve} - \theta)^+ \,dx - \int_{\RR^d} (H(\theta) - \ve\Delta \theta)\sgn_\delta(u_{n,\ve}-\theta)^+ \,dx.
\endaligned
\ee
Recall that $H$ is the hyperbolic operator defined in \eqref{010}. Consider the second and third terms on the left-hand side. We have, in exactly the same way as was done to prove the estimate \eqref{I131},
\[
\aligned
\int_{\RR^d} ( f(u_{n,\ve}) &- f(\theta)) \cdot  \nabla \sgn_\delta(u_{n,\ve}-\theta)^+ \,dx 
\\
& \qquad - \ve \int_{\RR^d} \sgn_\delta'(u_{n,\ve}-\theta)^+ |\nabla(u_{n,\ve}-\theta)|^2 \,dx 
\\
& \le \frac{M^2}{2\ve} \int_{\RR^d} ((u - v)^+)^2   \sgn_\delta'(u-v)  \,dx 
\\
&\qquad - \frac\ve2  \int_{\RR^d}   | \nabla (u - v) |^2  \sgn_\delta'(u-v)  \,dx
\\
& \le \frac{M^2}{2\ve} \int_{\RR^d} ((u - v)^+)^2  \sgn_\delta'(u-v)  \,dx .
\endaligned
\]
Now, for each $t>0$, the family of functions of $x\in \RR^d,$ 
\[
\zeta_\delta (x)  := \big( (u(t,x) - v(t,x))^+\big)^2 \sgn_\delta'(u(t,x)-v(t,x))^+ 
\]
satisfies for almost every $x\in \RR^d,$
\[
\aligned
0 \le \zeta_\delta (x) & = ((u - v)^+)^2  \frac1\delta \chi_{\{ (u-v)^+ \le \delta \}} 
\\
& \le (u- v)^+  \chi_{\{ (u-v)^+ \le \delta \}} 
\\
&\le  (u- v)^+ \in L^1(\RR^d)
\endaligned
\]
and
\[
\aligned
\zeta_\delta (x) \le \delta \to 0.
\endaligned
\]
Thus, by Lebesgue's theorem, we have
\be
\label{445.5}
\aligned
\int_{\RR^d} ( f(u_{n,\ve}) &- f(\theta)) \cdot  \nabla \sgn_\delta(u_{n,\ve}-\theta)^+ \,dx 
\\
& \qquad - \ve \int_{\RR^d} \sgn_\delta'(u_{n,\ve}-\theta)^+ |\nabla(u_{n,\ve}-\theta)|^2 \,dx 
\\
&\quad \le \frac{M^2}{2\ve} \int_{\RR^d} ((u - v)^+)^2  \sgn_\delta'(u-v)  \,dx \to 0, && \delta \to 0.
\endaligned
\ee
Let us take the limit $\delta \to 0$ in \eqref{445}. 
By Lebesgue's Theorem and \eqref{445.5}, the right-hand side converges to 
\[
\aligned
 &n  \int_{\RR^d} u_{n,\ve} \sgn(u_{n,\ve}-\theta)^+ \,dx  \cdot \int_{\RR^d} (u_{n,\ve} - \theta)^+ \,dx 
\\
& \qquad- n \int_{\RR^d} (u_{n,\ve} - \theta)^+ \,dx - \int_{\RR^d} (H(\theta) - \ve\Delta \theta)\sgn_(u_{n,\ve}-\theta)^+ \,dx,
\endaligned
\]
while the left-hand side converges to $\frac d{dt} \int_{\RR^d} (u_{n,\ve}-\theta)^+ \,dx $, as can be seen by writing
\[
\aligned
\frac d{dt} \int_{\RR^d} I_\delta(u_{n,\ve}-\theta)^+ \,dx = \int_{\RR^d} \sgn_\delta(u_{n,\ve}-\theta)^+ \del_t ( u_{n,\ve} - \theta) \,dx 
\endaligned
\]
and applying Lebesgue's theorem. We arrive at
\be
\label{446}
\aligned
&\frac d{dt} \int_{\RR^d} (u_{n,\ve}-\theta)^+ \,dx 
= n  \int_{\RR^d} u_{n,\ve} \sgn(u_{n,\ve}-\theta)^+ \,dx  \cdot \int_{\RR^d} (u_{n,\ve} - \theta)^+ \,dx 
\\
& \qquad- n \int_{\RR^d} (u_{n,\ve} - \theta)^+ \,dx - \int_{\RR^d} (H(\theta) - \ve\Delta \theta)\sgn(u_{n,\ve}-\theta)^+ \,dx.
\endaligned
\ee

Now define
\be
\label{446.5}
\varphi(t) = 
\int_{\RR^d} (u_{n,\ve}-\theta)^+(t) \,dx.
\ee
Then, \eqref{446} becomes
\be
\label{447}
\aligned
\varphi'(t)  &\le - n \varphi(t) \Big( 1 - \int_{\RR^d} u_{n,\ve} \sgn(u_{n,\ve}-\theta)^+ \,dx \Big) 
\\
&\quad - \int_{\RR^d} (H(\theta) - \ve\Delta \theta)\sgn(u_{n,\ve}-\theta)^+ \,dx.
\endaligned
\ee
We now use the key property \eqref{410} from Lemma \ref{3700} and the unit integral property \eqref{410}.
We have that 
\[
\aligned
1 - \int_{\RR^d} u_{n,\ve} \sgn(u_{n,\ve}-\theta)^+ \,dx
&= \int_{\RR^d}u_{n,\ve} \,dx - \int_{\RR^d} u_{n,\ve} \chi_{\{ u_{n,\ve} > \theta\}} \,dx
\\
&= \int_{\RR^d} u_{n,\ve} \chi_{\{ u_{n,\ve} \le \theta\}} \,dx
\\
&\ge \int_{\RR^d} u_{n,\ve} \chi_{\{ u_{n,\ve}  < \theta\}} \,dx
\endaligned
\]
and from \eqref{410}, 
$$ \int_{\RR^d} u_{n,\ve} \chi_{\{ u_{n,\ve}  < \theta\}} \,dx \ge \alpha.$$
This, $-H \le H^-$, and \eqref{447} give
\[
\aligned
\varphi'(t)  
& \le - n \varphi(t)  \int_{\RR^d} u_{n,\ve} \chi_{\{ u_{n,\ve}  < \theta\}} \,dx
\\
&\quad - \int_{\RR^d} (H(\theta) - \ve\Delta \theta)\sgn(u_{n,\ve}-\theta)^+ \,dx
\\
&\le - \alpha n \varphi(t) + \int_{\RR^d} ( H(\theta)^- + \ve\Delta \theta)\sgn(u_{n,\ve}-\theta)^+ \,dx.
\endaligned
\]
Thus, if $\ve \le 1$ and
\[
\aligned
C_\theta : = \sup_{t \in [0,T]} \int_{\RR^d} | H(\theta(t))^- | + | \Delta \theta (t) |  \,dx
\endaligned
\]
(note that $C_\theta$ depends on $T$ but does not depend on $n$ or $\ve$), then
\be
\label{448}
\aligned
\varphi'(t) \le  - \alpha n \varphi(t) + C_\theta,
\endaligned
\ee
or
\[
\aligned
&\big( e^{\alpha n t} \varphi(t) \big)' \le C_\theta e^{\alpha n t} 
\\
\Rightarrow{} & \varphi(t) \le C_\theta \int_0^t e^{\alpha n (s-t)} \,ds
\le C_\theta\frac{1- e^{-\alpha nt}}{\alpha n}
\le \frac{C_\theta}{\alpha n},
\endaligned
\]
which proves the estimate \eqref{400}, or rather, a slightly more precise version of \eqref{400} ensuring that $\varphi(t) \to 0$ as $t\to 0$.

We will now use \eqref{400} to prove the pointwise estimate \eqref{409}. Let 
\[
\aligned
m(t) :=   \| u_0 \|_{ L^\infty(\RR^d)} e^{t\frac{C_\theta}\alpha}.
\endaligned
\]
In parallel to \eqref{445}, it is easy to see that by adding and subtracting the appropriate terms in \eqref{318} and using $\sgn_\delta( u_{n,\ve} - m)^+$ as a test function, we have
\be
\label{450}
\aligned
&\frac d{dt} \int_{\RR^d} I_\delta(u_{n,\ve}-m)^+ \,dx - \int_{\RR^d} ( f(u_{n,\ve}) - f(m)) \cdot  \nabla \sgn_\delta(u_{n,\ve}-m)^+ \,dx  
\\
&\qquad+ \ve \int_{\RR^d} \sgn_\delta'(u_{n,\ve}- m )^+ |\nabla(u_{n,\ve} - m)|^2 \,dx 
\\
&\quad= n \int_{\RR^d} (u_{n,\ve} - \theta)^+ \,dx  \cdot \int_{\RR^d} u_{n,\ve} \sgn_\delta(u_{n,\ve}- m )^+ \,dx  
\\
& \qquad- n \int_{\RR^d} (u_{n,\ve} - \theta)^+ \sgn_\delta( u_{n,\ve} - m)^+ \,dx - \int_{\RR^d} m' \sgn_\delta(u_{n,\ve} - m)^+ \,dx.
\endaligned
\ee
By exactly the same reasoning that was done after \eqref{445}, we see that the two last terms on the left-hand side can be neglected. Also, the second term on the right-hand side  is nonpositive and so we discard it. For the first term in the right-hand side we have using \eqref{400},
\[
\aligned
n & \int_{\RR^d} (u_{n,\ve} - \theta)^+ \,dx  \cdot \int_{\RR^d} u_{n,\ve} \sgn_\delta(u_{n,\ve}- m )^+ \,dx   
\\
& \le \frac{C_\theta}\alpha \int_{\RR^d} u_{n,\ve} \sgn_\delta(u_{n,\ve}- m )^+ \,dx  
\\
& \le \frac{C_\theta}\alpha \int_{\RR^d} ( u_{n,\ve} - m)^+_\delta \,dx  +
\frac{C_\theta}\alpha \int_{\RR^d} m \sgn_\delta(u_{n,\ve}- m )^+ \,dx  .
\endaligned
\]
Therefore, \eqref{450} becomes after passing to the limit $\delta \to 0$ (recall that $m(t) =   \| u_0 \|_{ L^\infty}  e^{t\frac{C_\theta}\alpha}$),
\[
\aligned
\frac d{dt} \int_{\RR^d} (u_{n,\ve}-m)^+ \,dx 
& \le \frac{C_\theta}\alpha \int_{\RR^d} ( u_{n,\ve} - m)^+ \,dx  
\\
& \quad+ \int_{\RR^d} \big( - m' + \frac{C_\theta}\alpha  m \big) \sgn (u_{n,\ve}- m )^+ \,dx  
\\
& = \frac{C_\theta}\alpha \int_{\RR^d} ( u_{n,\ve} - m)^+ \,dx  .
\endaligned
\]
Since $(u_{0}(x) -m(0))^+ \equiv 0$ for all $x\in \RR^d$, integrating the previous estimate and using Gronwall's lemma gives that $( u_{n,\ve} - m)^+ = 0$ of almost all $t,x$, which is precisely the $L^\infty$ estimate \eqref{409}.
This completes the proof of Theorem~\ref{4000}.
\end{proof}

\proof[Proof of Theorem \ref{6000}]

Before establishing the uniform estimates, it is necessary to prove that $ u_{n,\ve}$ has the necessary smoothness for the calculations to be justified. For this we will repeatedly use \cite[Theorem 1.5, p.55]{GR}, which is a standard regularity result for the solution to a heat equation with right-hand side in $L^2$.

So, according to \cite[Theorem 1.5, p.55]{GR}, we see that for each $\ve,n$ the function $u_{n,\ve}$, solution of \eqref{318}, satisfies
\be
\label{453}
\aligned
u_{n,\ve} \in L^2 \big(0,T; H^2(\RR^d) \big), \quad \del_t u_{n,\ve} \in L^2(0,T; L^2(\RR^d)),
\endaligned
\ee
as long as $f$ is a $C^1$ function and the initial datum $u_0$ is in $H^1(\RR^d)$. This allows us to write the equation \eqref{318} in strong form,
\be
\label{500}
\aligned
&\del_t u_{n,\ve} + \dive f(u_{n,\ve}) - \ve \Delta u_{n,\ve} = n\, u_{n,\ve}\, \int_{\RR^d} (u_{n,\ve} - \theta)^+ \,dx
- n (u_{n,\ve} - \theta)^+,
\\
&u_{n,\ve}(0,x) = u_0(x).
\endaligned
\ee
We need further regularity for the time and space derivatives. Define $v = \del_t u_{n,\ve}$. Differentiate equation \eqref{500} in $t$ to get
\be
\label{502}
\aligned
\del_t v + \dive f'(u_{n,\ve}) v - {}&\ve \Delta v = n\, v \, \int_{\RR^d} (u_{n,\ve} - \theta)^+ \,dx
\\
& \quad+ n\, u \, \del_t\Big( \int_{\RR^d} (u_{n,\ve} - \theta)^+ \,dx\Big)
- n \del_t (u_{n,\ve} - \theta)^+.
\endaligned
\ee
Clearly, the first and third terms on the right-hand side are in $ L^2(0,T; L^2)$, in view of \eqref{453} and \eqref{400}. For the second term, recalling \eqref{446.5} and \eqref{448} gives 
\[
\aligned
\del_t\Big( \int_{\RR^d} (u_{n,\ve} - \theta)^+ \,dx\Big) \le C,
\endaligned
\]
while from \eqref{446} we find using \eqref{400} and the regularity of $\theta$
\[
\aligned
\del_t\Big( \int_{\RR^d} (u_{n,\ve} - \theta)^+ \,dx\Big) &\ge - n  \int_{\RR^d} (u_{n,\ve} - \theta)^+ \,dx
\\
& \quad- \int_{\RR^d} (H(\theta) - \ve\Delta \theta)\sgn_\delta(u_{n,\ve}-\theta)^+ \,dx
\\
& \ge - C_\theta.
\endaligned
\]
The previous two inequalities give $\Big| \del_t\Big( \int_{\RR^d} (u_{n,\ve} - \theta)^+ \,dx\Big) \Big| \le C$. Going back to \eqref{502}, we see that the right-hand side is in $ L^2(0,T; L^2)$. Applying \cite[Theorem 1.5, p.55]{GR}, we conclude that for each $n,\ve$
\be
\label{505}
\aligned
\del_t u_{n,\ve} \in L^2 \big(0,T; H^1(\RR^d) \big),
\endaligned
\ee
as long as $u_0 \in H^2(\RR^d)$ and $f\in (C^2(\RR))^d$.

Now we deduce some additional spatial regularity. Differentiate \eqref{500} in the direction $x_i$, $i=1,\dots,d$ to find
\be
\label{510}
\aligned
\del_t \del_i u_{n,\ve}  - \ve \Delta \del_i u_{n,\ve} &= n\, \del_i u_{n,\ve}\, \int_{\RR^d} (u_{n,\ve} - \theta)^+ \,dx
\\
& \quad- n \del_i (u_{n,\ve} - \theta)^+ 
- \dive (\del_i f(u)).
\endaligned
\ee
All the terms on the right-hand side are easily seen to be in $ L^2(0,T; L^2),$ from the regularity property \eqref{453}. Again invoking \cite[Theorem 1.5, p.55]{GR}, we see that (assuming $u_0 \in H^2(\RR^d)$)
\be
\label{515}
\aligned
u_{n,\ve} \in L^2 \big(0,T; H^3(\RR^d) \big).
\endaligned
\ee

We now have enough smoothness to rigorously proceed with the uniform estimates and prove Theorem~\ref{6000}.
We begin with a uniform estimate of $\|\nabla u_{n,\ve}\|_{L^1(\RR^d)}$. 
In the following calculation, for the sake of brevity, we omit the regularization parameter $\delta$ of the sign function.
Differentiate \eqref{500} in the direction $x_i$, $i=1,\dots,d$.
After summing and subtracting the appropriate terms, we find
\[
\aligned
&\del_t\del_i(u_{n,\ve} -\theta) +  \dive \big(f'(u_{n,\ve}) \del_i u_{n,\ve} - f'(u_{n,\ve}) \del_i \theta \big) -\ve \del_i \Delta (u_{n,\ve} - \theta ) 
\\
&\qquad=
n \, \del_i u_{n,\ve} \int_{\RR^d} (u_{n,\ve} - \theta)^+ \,dx 
- n \, \del_i(u_{n,\ve} -\theta)^+ + G,
\endaligned
\]
where
\be
\label{GG}
\aligned
G = - \del_t \del_i \theta - \dive \big( f'(u_{n,\ve}) \del_i \theta \big) - \ve \del_i \Delta \theta.
\endaligned
\ee
After multiplying by $\sgn(\del_i(u_{n,\ve} -\theta))$ and integrating on $\RR^d$ we find
\be
\label{516}
\aligned
\int_{\RR^d} \del_t |\del_i(u_{n,\ve} -{}&\theta)| \,dx =  - \int_{\RR^d}  \dive \big(f'(u_{n,\ve}) \del_i (u_{n,\ve} - \theta)  \big) \sgn\del_i(u_{n,\ve} -\theta) \,dx 
\\
&\quad +  \int_{\RR^d} \ve \del_i \Delta (u_{n,\ve} - \theta) \big) \sgn\del_i(u_{n,\ve} -\theta) \,dx
\\
&\quad+
n \Big( \int_{\RR^d} \del_i u_{n,\ve} \sgn\del_i(u_{n,\ve} -\theta) \,dx \Big) \Big( \int_{\RR^d} (u_{n,\ve} - \theta)^+ \,dx \Big)
\\
& \quad- n \int_{\RR^d} |\del_i(u_{n,\ve} -\theta)^+| \,dx + \int_{\RR^d} G \sgn \del_i(u_{n,\ve} -\theta) \,dx.
\endaligned
\ee
Consider the first two terms on the right-hand side. We integrate by parts and proceed as in the proof of estimates \eqref{I131} and \eqref{445.5}. They become using \eqref{MM}
\[
\aligned
&  \int_{\RR^d}   f'(u_{n,\ve}) \del_i (u_{n,\ve} - \theta)  \cdot \nabla \del_i(u_{n,\ve} -\theta) \sgn ' \del_i(u_{n,\ve} -\theta) \,dx 
\\
&\quad -  \int_{\RR^d} \ve | \nabla \del_i (u_{n,\ve} - \theta) |^2  \sgn ' \del_i(u_{n,\ve} -\theta) \,dx
\\
& \le \frac{M^2}{2\ve}  \int_{\RR^d}  ( \del_i (u_{n,\ve} - \theta))^2   \sgn '\del_i(u_{n,\ve} -\theta) \,dx 
\\
&\quad - \frac\ve2 \int_{\RR^d}  | \nabla \del_i (u_{n,\ve} - \theta) |^2  \sgn ' \del_i(u_{n,\ve} -\theta) \,dx.
\endaligned
\]
Now, as in  \eqref{I131} and \eqref{445.5}, but with $\del_i (u_{n,\ve} - \theta)$ instead of $ (u_{n,\ve} - \theta ) $, the first term above tends to zero as the regularization parameter of the sign function tends to zero, while the second term is nonpositive and so can be neglected. 

Going back to \eqref{516}, using the estimate \eqref{400}, the third term on the right-hand side is bounded by
\[
\aligned
\frac{C_\theta}\alpha \int_{\RR^d} | \del_i u_{n,\ve} | \,dx .
\endaligned
\]
The fourth term in \eqref{516} is nonpositive, while for the last one we have (recall \eqref{GG})
\be
\label{517}
\aligned
\int_{\RR^d} G \sgn \del_i(u_{n,\ve} -\theta) \,dx &\le 
\int_{\RR^d} |\del_t \del_i \theta |  +  |f''(u_{n,\ve})  \cdot \nabla u_{n,\ve} \del_i \theta| 
\\
& \quad+
| f'(u_{n,\ve}) \nabla \del_i \theta |  + \ve |\del_i \Delta \theta | \,dx
\\
&\le C_\theta \Big( M'  \int_{\RR^d} | \nabla u_{n,\ve}| \,dx + M + \ve +1 \Big) 
\\
& \le C_{\theta,f} \Big( \int_{\RR^d} | \nabla u_{n,\ve}| \,dx + 1 \Big)
\endaligned
\ee
for some $C_{\theta,f}$ depending on $\| \nabla_{t,x} \theta \|_{W^{2,1} ([0,+\infty) \times \RR^d)}$, $\| f' \|_\infty$ and $\| f'' \|_\infty$, but not on $\ve$.

Putting all the previous estimates together and integrating on $[0,t]$, we get from \eqref{516}
\[
\aligned
\int_{\RR^d}  |\del_i(u_{n,\ve} -\theta)| \,dx 
&\le
\int_{\RR^d}  |\del_i(u_0 -\theta(0))| \,dx 
\\
&\quad+
C_{\theta,f,\alpha} \int_0^t \Big(  \int_{\RR^d} | \nabla u_{n,\ve} | \,dx  + 1  \Big)\,dt ,
\endaligned
\]
where now $C_{\theta,f, \alpha} = (1 +  \frac1{\alpha}) C_{\theta,f}.$
Finally, writing $|\del_i u_{n,\ve}| \le |\del_i (u_{n,\ve}-\theta)| + |\del_i \theta|$ we find
\[
\aligned
\int_{\RR^d}  |\del_i u_{n,\ve}| \,dx 
&\le
\int_{\RR^d}  |\del_i u_0| \,dx +
C_{\theta,f,\alpha} \int_0^t \Big(  \int_{\RR^d} | \nabla u_{n,\ve} | \,dx  + 1  \Big)\,dt 
\\
&\quad + \int_{\RR^d} |\del_i\theta(0)| + |\del_i\theta|   \,dx.
\endaligned
\]
Since the last integral can be bounded by a constant $C_\theta(t)$ independent of $\ve$ or $n$, (recall the smoothness assumptions on $\theta$ in \eqref{50}) we find after applying Gronwall's inequality  that
\[
\aligned
\| \nabla u_{n,\ve}(t) \|_{L^1(\RR^d)} \le 
\big(\| \nabla u_0 \|_{L^1(\RR^d)} + C_\theta(t)\big) e^{tC},
\endaligned
\]
for some constant $C$ depending on $\theta$, $f'$, $f''$, $\alpha$ and $d$, but independent of $\ve$ and $n$. Thus,
\be
\label{520}
\aligned
\nabla u_{n,\ve} \in L^\infty(0,T; L^1(\RR^d))\qquad \text{uniformly in }n,\ve.
\endaligned
\ee

Next, we obtain a uniform estimate of $\| \del_t u_{n,\ve}(t)\|_{L^1(\RR^d)}$. Differentiate the equation \eqref{500} in $t$ (after adding and subtracting appropriate terms) to get
\[
\aligned
&\del_{tt}(u_{n,\ve} -\theta) +  \dive \big(f'(u_{n,\ve}) \del_t u_{n,\ve} - f'(u_{n,\ve}) \del_t \theta \big) -\ve \del_t \Delta (u_{n,\ve} - \theta ) 
\\
&\quad=
n \, \del_t u_{n,\ve} \int_{\RR^d} (u_{n,\ve} - \theta)^+ \,dx +
n \,  u_{n,\ve} \int_{\RR^d} \del_t (u_{n,\ve} - \theta)^+ \,dx 
\\
& \qquad - n \, \del_i(u_{n,\ve} -\theta)^+ + K,
\endaligned
\]
where
\be
\label{KK}
\aligned
K = - \del_{tt} \theta - \dive \big( f'(u_{n,\ve}) \del_t \theta \big) - \ve \del_t \Delta \theta.
\endaligned
\ee
Now multiply by $\sgn\del_t(u_{n,\ve} -\theta)$ and integrate on $\RR^d$. The flux and viscosity terms give a nonpositive contribution on the the right-hand side, exactly as in the previous estimate, so we omit their treatment. We find
\be
\label{525}
\aligned
\frac d{dt} \int_{\RR^d} | \del_t (u_{n,\ve} &-\theta)| \,dx \le
n \Big( \int_{\RR^d} \del_t u_{n,\ve} \sgn \del_t(u_{n,\ve} -\theta) \,dx \Big) \Big( \int_{\RR^d} (u_{n,\ve} - \theta)^+ \,dx \Big)
\\
&\quad
+n \Big( \int_{\RR^d} u_{n,\ve} \sgn \del_t(u_{n,\ve} -\theta) \,dx \Big) \Big(\int_{\RR^d} \del_t (u_{n,\ve} - \theta)^+ \,dx \Big)
\\
&\quad- n \int_{\RR^d} \del_t (u_{n,\ve} -\theta)^+  \sgn \del_t(u_{n,\ve} -\theta)\,dx
\\
&\quad+ \int_{\RR^d} |K | \,dx.
\endaligned
\ee
The first line on the right-hand side is estimated using \eqref{400},
\[
\aligned
n \Big( \int_{\RR^d} \del_t u_{n,\ve} \sgn \del_t(u_{n,\ve} -\theta) \,dx \Big) \Big( \int_{\RR^d} (u_{n,\ve} - \theta)^+ \,dx \Big) \le  \frac{C_\theta}\alpha \int_{\RR^d} |\del_t u_{n,\ve}| \,dx .
\endaligned
\]
Now we consider the second and third lines on the right-hand side of \eqref{525}. Using $\int_{\RR^d} u \sgn \del_t(u -\theta) \,dx \le 1$, $\int_{\RR^d} u_{n,\ve} \,dx =1$, and $\del_t v^+ = \sgn v^+ \del_tv$,  we find that these terms are bounded in the following way:
\[
\aligned
n  \int_{\RR^d} \del_t & (u_{n,\ve} - \theta)^+ \,dx 
- n \int_{\RR^d} \del_t (u_{n,\ve} -\theta)^+  \sgn \del_t(u_{n,\ve} -\theta)\,dx
\\
& = n \int_{\RR^d} \sgn (u_{n,\ve} -\theta)^+ \Big ( \del_t (u_{n,\ve} -\theta) - |\del_t (u_{n,\ve} -\theta)| \Big) \,dx
\\
& \quad \le 0.
\endaligned
\]
The last term in \eqref{525} is estimated exactly as in \eqref{517} to give
\[
\aligned
\int_{\RR^d} |K| \,dx  & \le C_{\theta,f} \Big( \int_{\RR^d} | \del_t u_{n,\ve}| \,dx + 1 \Big).
\endaligned
\]

Using the foregoing estimates, \eqref{525} becomes upon integration on $(0,t)$,
\[
\aligned
 \int_{\RR^d} | \del_t (u_{n,\ve} -\theta)| \,dx 
&\le
\int_{\RR^d}  |\del_t(u_0 -\theta(0))| \,dx 
\\
&\quad+
C_{\theta,f,\alpha} \int_0^t \Big(  \int_{\RR^d} | \del_t u_{n,\ve} | \,dx  + 1  \Big)\,dt ,
%
%
\endaligned
\]
with $C_{\theta,f, \alpha} = (1 +  \frac1{\alpha}) C_{\theta,f}.$
Therefore, we obtain
\[
\aligned
\int_{\RR^d} | \del_t u_{n,\ve}| \,dx &\le
\int_{\RR^d} | \del_t u_0 | \,dx +
C_{\theta,f,\alpha} \int_0^t \Big(  \int_{\RR^d} | \del_t u_{n,\ve} | \,dx  + 1  \Big)\,dt 
\\
& \quad+ \int_{\RR^d} |\del_t \theta|  + |\del_t \theta(0)| dx.
\endaligned
\]
Now, using the equation \eqref{500} one obtains
\be
\label{530}
\aligned
\int_{\RR^d} | \del_t u_0 | \,dx &\le M \| \nabla u_0 \|_{L^1(\RR^d)}
+ \ve \| \Delta u_0\|_{L^1(\RR^d)}.
\endaligned
\ee
As in \cite[p.68]{GR}, we consider a smoothing of $u_0$ such that 
$ \ve \| \Delta u_0\|_{L^1(\RR^d)} \le C \| \nabla u_0\|_{L^1(\RR^d)}$ for some universal constant depending only on the dimension $d$. Thus
\[
\aligned
\| \del_t u_{n,\ve} \|_{L^1(\RR^d)} &\le
C \| \nabla u_0\|_{L^1(\RR^d)} +
C_{\theta,f,\alpha} \int_0^t \| \del_t u_{n,\ve} \|_{L^1(\RR^d)}  + 1 \,dt 
\\
&\quad
+ \|\del_t \theta\|_{L^1(\RR^d)}  + \|\del_t \theta(0)\|_{L^1(\RR^d)}
\endaligned
\]
and so applying Gronwall's lemma gives
\[
\aligned
\| \del_t u_{n,\ve}(t) \|_{L^1(\RR^d)} &\le \big( C \| \nabla u_0\|_{L^1(\RR^d)} +
C_\theta(t) \big) e^{tC}.
\endaligned
\]
This concludes the proof of Theorem \ref{6000}.
\endproof


\section{Solvability of the obstacle-mass constraint problem}
\label{S5000}
In this section, we establish existence of an entropy solution for problem \eqref{010}--\eqref{012}, in the sense of Definition~\ref{DEFSOL}, by the vanishing viscosity method. 

%
%

\proof[Proof of Theorem \ref{TEO}]
1. First, for $\ve> 0$ and $n \in \NN$ we consider the nonlocal penalized viscous problem \eqref{315}, which we repeat here for convenience:
$$
\aligned
&\del_t u_{n,\ve} + \dive f(u_{n,\ve}) - \ve \Delta u_{n,\ve} = n\, u_{n,\ve}\, \int_{\RR^d} (u_{n,\ve} - \theta)^+
- n (u_{n,\ve} - \theta)^+,
\\
&u_{n,\ve}(0,x) = u_0(x).
\endaligned
$$
For $\vp \in C^\infty_c((-\infty,T) \times \RR^d)$ and $\eta$ an entropy (assumed $C^2$ without loss of generality), multiply 
\eqref{315} by $\vp \, \eta'(u_{n,\ve}- k \theta)$ and integrate in $(0,T) \times \RR^d =: \Pi_T$. We obtain
  $$
    \begin{aligned}
    &- \iint_{\Pi_T} \eta(u_{n,\ve}- k \theta) \, \vp_t \, \,dx \,dt + \iint_{\Pi_T} \eta'(u_{n,\ve}- k \theta) \, \vp \, \partial_t(k  \theta) \, \,dx \,dt
    \\[5pt]
    &-\iint_{\Pi_T}  \eta'(u_{n,\ve}- k \theta) \big( f(u_{n,\ve}) - f(k \theta) \big) \cdot \nabla \vp \, \,dx \,dt 
    \\[5pt]
    &+ \iint_{\Pi_T}  \eta'(u_{n,\ve}- k \theta) \, \vp  \, \dive f(k \theta) \, \,dx \,dt
    \\[5pt]
    & -\iint_{\Pi_T} \ve \, \Delta \eta(u_{n,\ve}- k \theta) \, \vp \, \,dx \,dt -\iint_{\Pi_T} \ve \, \Delta (k \theta) \,  \eta'(u_{n,\ve}- k \theta) \, \vp \, \,dx \,dt
    \\[5pt]
& - \int_{\RR^d} \eta( u_0(x) - k \,  \theta(0,x) ) \, \vp(0,x) \, dx
	\\[5pt]
    =& \iint_{\Pi_T}  \big( f(u_{n,\ve}) - f(k \theta) \big) \cdot \nabla \big(\eta'(u_{n,\ve}- k \theta) \big)  \,\vp   \,dx \,dt   
    \\[5pt]
    &-\iint_{\Pi_T} \ve \, \eta''(u_{n,\ve}- k \theta) \,  |\nabla(u_{n,\ve} - k \theta)|^2\, \vp \, \,dx \,dt 
    \\[5pt]
    &\quad + \iint_{\Pi_T}  \Big( n \, u_{n,\ve} \int (u_{n,\ve}-\theta)^+ \, dx - n (u_{n,\ve} - \theta)^+ \Big) \,  \eta'(u_{n,\ve}- k \theta) \, \vp \, \,dx \,dt. 
    \end{aligned}
$$  
Let $\eta(u)$ be an approximation (uniform on compact sets) of the Kruzkov entropy $|u|$. Then, similarly to the estimate \eqref{I13}, the first and second terms on the right-hand give nonpositive or vanishing contributions as we take the limit in that approximation (see \eqref{I13} for a totally similar procedure).
Thus,
neglecting the negative terms on the right-hand side, it follows that in the sense of distributions
\be
\label{650}
  \begin{aligned}
   \partial_t  |u_{n,\ve}- k \theta|  &+ \dive \Big( \sgn(u_{n,\ve}- k \theta) \big(f(u_{n,\ve}) - f(k \theta) \big)\Big) - \ve \Delta  |u_{n,\ve}- k \theta|
   \\[5pt]
   &\leq n \, u_{n,\ve} \,  \sgn(u_{n,\ve}- k \theta) \int_{\RR^d} (u_{n,\ve} - \theta)^+ \, dx 
   \\[5pt]
   & \qquad-  \sgn(u_{n,\ve}- k \theta)  \big( H(k \theta) - \ve \Delta(k \theta) \big),
   \end{aligned}
\ee
which incidentally motivates the precise formulation in Definition~\ref{DEFSOL}.

\medskip
2.  Now, we define for almost all $t \in (0,T)$, 
$$
\lambda_{n,\ve}(t) := n\int_{\RR^d} (u_{n,\ve}(t) - \theta(t))^+ \,dx.
$$ 
According to the estimate \eqref{400}, we have that $\lambda_{n,\ve}(t)$ is uniformly bounded for a.a.~$t \in (0,T)$. Thus (if necessary taking a subsequence), $\lambda_{n,\ve}(t)$ converges weak-star in $L^\infty(0,T)$ to some $\lambda(t)$ as $n \to \infty$, $\ve \to 0$.

\medskip
3. With the inequality \eqref{650} in hand, and the estimates collected in previous sections, it is a standard matter to pass to limit as $n,\ve \to 0$ and obtain an entropy solution. Indeed, using standard compactness results (see, e.g., the totally similar procedure in \cite[p.70]{GR}), the family $(u_{n,\ve})$ has a subsequence (which we do not relabel) converging a.e. on $\Pi_T$ and in $L^1_\loc((0,T)\times \RR^d)$ to some $u \in L^{\infty}((0,T)\times \RR^d)$. The gradient estimate in Theorem~\ref{6000} ensures that $u(t) \in BV(\RR^d)$ for a.a. $t \in (0,T)$. Note that Theorem~\ref{6000} requires that the initial datum is smooth enough, so we use a mollification of $u_0$ depending on $\ve$. The procedure to obtain $u_0$ in the limit is exactly the same as in \cite{GR}, so we omit it for the sake of simplicity.
 Moreover, from item 2 we see that the first term on the right-hand side of \eqref{650} converges to $ u \lambda(t) \eta'(u-k\theta)$ weak-star in  $L^\infty\big( (0,T)\times \RR^d \big)$, which is enough to pass to the limit on \eqref{650}. Thus $(u,\lambda)$ is a solution of problem~\eqref{010}--\eqref{012} according to Definition~\ref{DEFSOL}. This completes the proof of Theorem~\ref{TEO}.
\endproof


\section*{Appendix}
The goal of this appendix is to prove Lemma \ref{3700}, which is crucial in our analysis. For its proof, we require the Lemma~\ref{420}, which we prove first.
This is a modification of a classical comparison result comparing the solution of the nonlocal problem \eqref{315} with 
the solution of the homogenous conservation law  \eqref{404}.
The key point is that this comparison property is independent of $n$, and that the nonlocal terms do not influence the result.

\begin{lemma}
\label{420}
Let $u_{n,\ve}$ be a solution of \eqref{315}, and let $v_{\ve}$ be a solution to the Cauchy problem for the viscous 
homogenous conservation law \eqref{404}. Then, $u_{n,\ve} \ge v_{\ve}$. In particular, this comparison property holds for all $n$.
\end{lemma}
\begin{proof}[Proof of the lemma.]
We drop the subscripts $n,\ve$ from $u_{n,\ve}$ and $\ve$ from $v_{\ve}$ during the proof.
Subtract \eqref{315} from \eqref{404}, multiply by $(v_{\ve} - u_{n,\ve})^+$, and integrate on $\RR^d$ to get (with $w = v_{\ve} - u_{n,\ve}$)
\[
\aligned
&\frac{d}{dt} \int_{\RR^d} (w^+)^2 \,dx \le - \int_{\RR^d} \dive \big(f(v) - f(u) \big)w^+ \,dx
+ \ve \int_{\RR^d} \Delta w\, w^+ \,dx 
\\
&\quad 
- n \Big( \int_{\RR^d} u \, w^+ \,dx \Big) \Big( \int_{\RR^d} (u - \theta)^+ \,dx \Big)
+ n \int_{\RR^d} (u - \theta)^+ w^+ \,dx.
\endaligned
\]
Now note that by the compatibility condition of Definition~\ref{COMP}, $v_\ve $ has initial datum $v_0 < \underline{\theta}$, so by \eqref{50.2} and the maximum principle for the problem \eqref{404}, we have that 
 $v_{\ve} \le \theta$ (see \cite{GR}). Therefore, in the set where $u_{n,\ve} \ge \theta$, then necessarily $u_{n,\ve} > v_{\ve},$ or $w^+ =0.$ Thus we conclude that
$$
n \int_{\RR^d} (u - \theta)^+ w^+ \,dx = 0.
$$ 
Also neglecting the nonpositive term on the right-hand, we find
\[
\aligned
&\frac{d}{dt} \int_{\RR^d} (w^+)^2 \,dx \le - \int_{\RR^d} \dive \big(f(v) - f(u) \big)w^+ \,dx
+ \ve \int_{\RR^d} \Delta w\, w^+ \,dx .
\endaligned
\]
Using the Lipschitz condition on $f$ and integration by parts,
we have
\[
\aligned
\frac{d}{dt} \int_{\RR^d} (w^+)^2 \,dx &\le    \int_{\RR^d}  \big(f(v) - f(u) \big) \cdot\nabla w^+ \,dx
- \ve \int_{\RR^d} \nabla w \nabla w^+ \,dx 
\\
& \le  M \int_{\RR^d} |w| |\nabla w^+| \,dx - \ve \int_{\RR^d} |\nabla w^+|^2\,dx 
\\
& =   M \int_{\RR^d} |w^+| |\nabla w^+| \,dx - \ve \int_{\RR^d} |\nabla w^+|^2\,dx .
\endaligned
\]
Then, using a weighted Young inequality, we easily find
\[
\aligned
\frac{d}{dt}\int_{\RR^d} (w^+)^2 \,dx &\le     \frac{ M^2}{2\ve}  \int_{\RR^d} |w^+|^2 \,dx 
 - \frac{ \ve}{2} \int_{\RR^d} |\nabla w^+|^2\,dx .
\endaligned
\]
Integrating on $[0,t]$ for $t\le T$, and using Gronwall's lemma, we conclude that
$\int_{\RR^d} (w^+)^2 \,dx = 0$ and so $v\le u$ on $[0,T]$. This proves Lemma~\ref{420}.
\end{proof}


\subsection*{Proof of Lemma~\ref{3700}}
One recalls that, 
its motivation was discussed in Remark~\ref{R2200} above and it is used in the proof of Theorem~\ref{4000}.
We repeat the statement here for convenience.

\medskip	
\emph{Under the same conditions of Theorem~\ref{4000}, there exists a constant $\alpha >0$ depending on $T$, $u_0$ and $\theta$, but not on $n$ or $\ve$, such that the estimate \eqref{410} is valid, that is to say:}
$$
\aligned
\inf_{0\le t \le T} \int_{\{ u_{n,\ve}  < \theta\}} u_{n,\ve}(t)  \,dx \ge \alpha.
\endaligned
$$
\medskip

The idea of the proof is the following: as discussed in Section~\ref{SS}, Definition~\ref{COMP} is designed to ensure that the support of $ u_{n,\ve}$ always travels into regions where 
the integral of $\theta$ is greater than one. In view of this, and the fact (established in \eqref{397}) that the total mass of $ u_{n,\ve}$ is one, necessarily $ u_{n,\ve}$ cannot have all its nonzero values above $\theta$, otherwise compatibility in the sense of Definition~\ref{COMP} would be violated. Therefore, $ u_{n,\ve}$ must retain some mass below $\theta$, which is the claim in \eqref{410}. We now make precise this statement, using a contradiction argument.

1. 
Suppose \eqref{410} is false. Then, there are sequences $\ve_j \to 0,$ $t_j \in (0,T],$ $n_j \to \infty$, such that 
\be
\label{545}
\aligned
     \int_{\RR^d} u_j\chi_{\{ u_j < \theta\}} (t_j) \,dx <  \frac1j,
\endaligned
\ee
where $u_j \in L^1(\RR^d)$ is the solution $u_{n_j,\ve_j}$ of equation \eqref{315} (in the sense of Theorem~\ref{3500}), at time $t_j$ (so, $u_j(x) = u_{n_j,\ve_j}(t_j,x)$). Upon extraction of a subsequence (which here, and in what follows, we do not relabel), we may suppose $t_j \to t^*$ for some $t^* \in (0, T]$ as $j\to \infty$. Observe that due to Remark~\ref{2000} we ensure that $t^*>0$.
Thus, if we set $ w_j(x) := u_j(x) \chi_{\{ u_j(x) <\theta(t_j,x) \}} \in L^1(\RR^d)$, then \eqref{545} gives $ w_j\to 0$ as $j\to \infty$ in $L^1(\RR^d),$ since $u_j$ is nonnegative.

Let $v_j \in L^1(\RR^d)$ denote the (smooth) solution of the viscous problem \eqref{404} 
with viscosity parameter $\ve=\ve_j,$ at time $t_j$. That is, $v_j(x) = v_{\ve_j}(t_j,x)$ in \eqref{404}.
According to the comparison Lemma~\ref{420}, we have $v_j \le u_j$, and so $v_j  \chi_{\{ u_j <\theta(t_j) \}} \le w_j$ for a.e. $x \in \RR^d$. From $w_j \to 0$ in $L^1(\RR^d)$ we obtain 
\be
\label{547}
\aligned
v_j  \chi_{\{ u_j <\theta(t_j) \}} \to 0 \text{ in }L^1(\RR^d)
\endaligned
\ee
as $j\to \infty$.

\medskip
2.  We have $v_j(t_j) \to v(t^*)$ in $L^1(\RR^d)$ as $j\to \infty$, with $v$ solving \eqref{403}. Indeed, according to standard results concerning the vanishing viscosity approximation of hyperbolic conservation laws and the continuity in time of the viscous approximations (see, for instance, \cite{GR}), we have
\[
\aligned
\| v_j(t_j) - v(t^*)\|_{L^1(\RR^d)} \le \| v_j(t_j) - v_j(t^*)\|_{L^1(\RR^d)} + \| v_j(t^*) - v(t^*)\|_{L^1(\RR^d)} \to 0
\endaligned
\]
as $j\to \infty$.
Also, we have $ \chi_{\{ u_j(x) <\theta(t_j,x) \}} \ws \xi$ in 
$ L^\infty(\RR^d)$, for some $\xi\in L^\infty(\RR^d)$. Thus, \eqref{547} gives $v(t^*) \xi = 0$ a.e. on $\RR^d.$ Therefore, 
\be
\label{549}
\aligned
\xi =0 \text{ a.e. on }\{x\in \RR^d | v(x, t^*)>0\},
\endaligned
\ee 
which we abbreviate to $\{ v(t^*) >0\}$. 
Now, observe that a sequence of nonnegative functions weakly converging to zero also converges strongly in $L^1_\loc$.
Since $\chi_{\{ u_j <\theta(t_j) \}} \ge0$, we conclude from $\chi_{\{ u_j <\theta(t_j) \}} \ws \xi$ and \eqref{549} that actually 
$$
\chi_{\{ u_j <\theta(t_j) \}} \to 0 \text{ strongly in }L^1_\loc(\{v(t^*)>0\})
$$ 
and a.e.~on $\{v(t^*)>0\}$, as $j\to \infty$.

\medskip
3. Let $B_R$ denote the ball of radius $R>0$ centered on the origin. Let $\delta> 0$ to be chosen later. According to Egorov's Theorem, There exists a set $J_\delta \subset (\{v(t^*)>0\} \cap B_R)$ such that $|J_\delta| \le \delta$ and $\chi_{\{ u_j <\theta(t_j) \}} \to 0$ uniformly on $ V_\delta := (\{v(t^*)>0\} \cap B_R) \setminus J_\delta$ as $j\to \infty$. Since $\chi_{\{ u_j <\theta(t_j) \}}$ only takes the values 0 and 1, 
this means that for sufficiently large $j$, we must have $u_j(x) > \theta(t_j,x)$ a.e.~on $V_\delta$. Therefore,
\be
\label{550}
\aligned
\int_{V_\delta} u_j(x) \,dx > \int_{V_\delta} \theta(t_j,x) \,dx
=\int_{\{v(t^*)>0\} \cap B_R}\!\!\!\!\!\!\!\!\!\!\!\! \theta(t_j,x) \,dx - \int_{J_\delta} \theta(t_j,x) \,dx.
\endaligned 
\ee
Now, from the compatibility condition \eqref{1.200}, we deduce that for large enough $R$, 
$$
\int_{\{v(t^*)>0\} \cap B_R} \!\!\!\!\!\!\!\!\!\!\!\!\theta(t^*,x) \,dx > 1+ \beta / 2,
$$ 
and, by the $L^1$ continuity property \eqref{50.1},
$$
\int_{\{v(t^*)>0\} \cap B_R} \!\!\!\!\!\!\!\!\!\!\!\!\theta(t_j,x) \,dx > 1+ \beta / 2
$$ 
for sufficiently large $j$.
On the other hand, from Lebesgue's theorem and \eqref{50.1}, we see that since $\theta$ is locally integrable, we have $\int_{J_\delta} \theta(t^*,x) \,dx \to 0$ when $\delta \to 0$. Therefore, we choose $\delta$ small enough such that
\[
\aligned
\int_{J_\delta} \theta(t^*,x) \,dx \le \frac{\beta}{8}.
\endaligned
\]
Again using \eqref{50.1}, we find for sufficiently large $j$
\[
\aligned
\int_{J_\delta} \theta(t_j,x) \,dx &\le \int_{B_R} |\theta(t_j,x) - \theta(t^*,x)| \,dx 
+ \frac{\beta}{8}
\\
&\le \frac{\beta}{8} + \frac{\beta}{8} = \frac{\beta}{4}.
\endaligned
\]

We conclude from \eqref{550} and from the unit integral property \eqref{397}  that 
\[
\aligned
1 \ge \int_{V_\delta} u_j(x) \,dx > 1 + \frac\beta2 - \frac\beta4 = 1 + \frac\beta4,
\endaligned
\]
which is a contradiction. Thus \eqref{545} cannot hold and so \eqref{410} is proven. This concludes the proof of Lemma~\ref{3700}. \qed

\section{Acknowledgments}

Paulo Amorim was partially supported by FAPERJ grants no. APQ1 - 111.400/2014 and {\bf Jovem Cientista do Nosso Estado} grant no. 202.867/2015, and CNPq grant no. 442960/2014-0.

Wladimir Neves is partially supported by
CNPq through the grants
484529/2013-7,  308652/2013-4, 
and by FAPERJ, {\bf Cientista do Nosso Estado},
through the grant E-26/203.043/2015.

\small 
\newcommand{\auth}{\textsc}

\end{document}